\documentclass[11pt]{amsart}
\usepackage{amsmath, amscd, amssymb}
\usepackage[enableskew,vcentermath]{youngtab}
\usepackage{ytableau}
\usepackage[mathscr]{eucal}
\usepackage[normalem]{ulem} 
\usepackage{soul} 
\usepackage[all,cmtip]{xy}
\usepackage{tikz-cd}
\usepackage{color} 
\usepackage{setspace}
\setulcolor{red}
\sethlcolor{yellow}
\setstcolor{blue}
\usepackage{amscd}
\usepackage{xfrac}

\usepackage{geometry}
\usepackage{amsfonts}
\usepackage[colorlinks,linkcolor=blue,citecolor=blue, pdfstartview=FitH]{hyperref}
\usepackage{graphicx}
\usepackage{cite}

\numberwithin{equation}{section}

\theoremstyle{definition}
\newtheorem{thm}{Theorem}[section]
\newtheorem{theorem}[thm]{Theorem}

\newtheorem{lemma}[thm]{Lemma}
\newtheorem{corollary}[thm]{Corollary}
\newtheorem{proposition}[thm]{Proposition}

\newtheorem{remark}[thm]{Remark}

\newtheorem{definition}[thm]{Definition}

\newtheorem{example}[thm]{Example}

\newtheorem{defn-thm}[thm]{Definition-Theorem}

\newtheorem{Notation}[thm]{Notation}

\newenvironment{observe}{\noindent\textcolor{blue}{\textit{Observation}}.}{\hfill \textcolor{blue}{$\blacktriangleleft$}\par}
\newtheorem*{theorem*}{Theorem}
\newtheorem*{proposition*}{Proposition}

\usepackage{todonotes}
\definecolor{pistachio}{rgb}{0.58, 0.77, 0.45}
\definecolor{eggshell}{rgb}{0.94, 0.92, 0.84}

\newcommand{\sE}{{\mathcal E}}

\newcommand{\sK}{{\mathcal K}}
\newcommand{\sL}{{\mathcal L}}

\newcommand{\sN}{{\mathcal N}}
\newcommand{\sO}{{\mathcal O}}
\newcommand{\sP}{{\mathcal P}}

\newcommand{\sR}{{\mathcal R}}

\newcommand{\cO}{{\mathcal O}}
\newcommand{\cP}{{\mathcal P}}

\newcommand{\g}{{\mathfrak g}}

\newcommand{\gc}{\mathfrak{c}}
\newcommand{\gp}{\mathfrak{p}}

\newcommand{\gn}{\mathfrak{n}}
\newcommand{\gl}{\mathfrak{l}}

\newcommand{\CC}{{\mathbb C}}

\newcommand{\N}{{\mathbb N}}

\newcommand{\Sp}{\operatorname{Sp}}

\newcommand{\SO}{\operatorname{SO}}
\newcommand{\OG}{\operatorname{OG}}

\newcommand{\Prym}{\operatorname{\mathbf{Prym}}}

\newcommand{\Ker}{\operatorname{Ker}}

\newcommand{\Pic}{\operatorname{Pic}}

\newcommand{\Spec}{\operatorname{Spec}}

\newcommand{\Res}{\operatorname{Res}}
\newcommand{\Jac}{{\operatorname{Jac}}}

\newcommand{\codim}{{\operatorname{codim}}}

\newcommand{\Tot}{{\operatorname{Tot}}}

\newcommand{\Frac}{{\operatorname{Frac}}}

\newcommand{\Fil}{\operatorname{Fil}}

\newcommand{\KL}{\operatorname{KL}}

\newcommand{\Ad}{\operatorname{Ad}}

\newcommand{\btheorem}{\begin{theorem}}
\newcommand{\etheorem}{\end{theorem}}
\newcommand{\bproposition}{\begin{proposition}}
\newcommand{\eproposition}{\end{proposition}}
\newcommand{\bdefinition}{\begin{definition}}
\newcommand{\edefinition}{\end{definition}}
\newcommand{\bcorollary}{\begin{corollary}}
\newcommand{\ecorollary}{\end{corollary}}
\newcommand{\bproof}{\begin{proof}}
\newcommand{\eproof}{\end{proof}}
\newcommand{\bremark}{\begin{remark}}
\newcommand{\eremark}{\end{remark}}
\newcommand{\eexample}{\end{example}}
\newcommand{\bexample}{\begin{example}}
\newcommand{\elemma}{\end{lemma}}
\newcommand{\blemma}{\begin{lemma}}
\newcommand{\bobserve}{\begin{observe}}
\newcommand{\eobserve}{\end{observe}}

\renewcommand{\bar}{\overline}

\renewcommand{\phi}{\varphi}
\newcommand{\ee}{\end{eqnarray*}}
\newcommand{\be}{\begin{eqnarray*}}
\newcommand{\Obar}{\overline{\bf{O}}}

\newcommand{\beq}{\begin{equation}}
\newcommand{\eeq}{\end{equation}}
\newcommand{\bd}{\begin{enumerate}}
\newcommand{\ed}{\end{enumerate}}
\newcommand{\bti}{\begin{tikzcd}}
\newcommand{\eti}{\end{tikzcd}}

\renewcommand{\tilde}{\widetilde}

\renewcommand{\bf}[1]{\mathbf{#1}}

\title[On the generic fibers and true base of parabolic $\SO_{2n}$-Hitchin systems]{On the generic fibers and true base of parabolic $\SO_{2n}$-Hitchin systems}

\author{Bin Wang}
\address{Department of Mathematics, Chinese University of Hong Kong, New Territories, Hong Kong SAR.}
\email{binwang@math.cuhk.edu.hk}

\author{Xueqing Wen}
\address{Chongqing University of Technology, No. 69, Hongguang Avenue, Banan District, Chongqing, 400054, China.}
\email{wenxq@cqut.edu.cn}

\author{Yaoxiong Wen}
\address{School of Mathematics, Korea Institute for Advanced Study, Seoul 02455, Korea.}
\email{y.x.wen.math@gmail.com}

\date{}

\begin{document}

\maketitle
\begin{abstract}
    In this paper, we confirm a physical conjecture regarding the parabolic $\SO_{2n}$-Hitchin system, showing that Hitchin map factors through a finite cover of the Hitchin base that is isomorphic to an affine space. We first show that the generic Hitchin fiber is disconnected, with the number of components determined by the degree of the generalized Springer map, and then construct the cover explicitly. To this end, we introduce and study a new class of moduli spaces, termed \emph{residually nilpotent Hitchin systems}, and analyze their generic Hitchin fibers. Furthermore, we uncover an interesting connection between self-duality of the generic Hitchin fiber and special nilpotent orbits.
\end{abstract}

\tableofcontents
\clearpage

\section{Introduction}

\subsection{Background}\label{subsec:backgroup} 

Moduli of Higgs bundles on curves along with the corresponding Hitchin maps form an algebraically completely integrable systems  and has played an important role in algebraic geometry and geometric representation theory. By Simpson~\cite{Sim90, Sim92}, \emph{parabolic} Hitchin systems serves as an analogue over punctured curves, or curves with marked points. Motivated by  Gukov, Kapustin and Witten~\cite{KW07, GW08, GW10}, we study a variant of the parabolic Hitchin system, namely the \emph{residually nilpotent} Hitchin system, which associates a nilpotent orbit closure to a Hitchin system. In \cite{WWW24}, we constructed and studied such moduli spaces for types B and C, developing local techniques to describe generic Hitchin fibers. In this paper, we turn our attention to type D, specifically the case $G = \SO_{2n}$, thereby completing the picture for all classical groups.\footnote{The type A case has been studied in \cite{She18, SWW22, She23, SWW22}.}

Another motivation for our work comes from Yun's global Springer theory \cite{Yun11}. Nilpotent orbits are \emph{classical} objects in geometric representation theory. The \emph{local} aspect-studying liftings of nilpotent elements in the loop Lie algebra-dates back to the foundational work of Kazhdan and Lusztig \cite{KL88}; see also \cite{Yun21}. However, the \emph{global} analogue, akin to Yun’s theory, which associates nilpotent orbits to a Hitchin system, has only been implicitly explored in the literature. We believe that the residually nilpotent Hitchin systems can fill this jigsaw. As we will see, the geometry of these systems is deeply intertwined with the representation-theoretic nature of nilpotent orbits, and offers new insights into the classical theory. 

The central theme of this paper is to investigate the geometry of generic Hitchin fibers and Hitchin bases. The study of generic fibers is strongly motivated by the mirror symmetry phenomenon observed in Hitchin systems, as explored in a series of influential works \cite[\textit{etc}]{HT03, DG02, DP12, Yun11, Yun12, CZ17, GWZ20, GWZ20g, MS21}. On the other hand, the geometry of the Hitchin base in type~D presents additional interest due to the presence of the Pfaffian, which leads to richer structure compared to other types. This aspect is further inspired by Tachikawa's conjectural description \cite{Tac} of the ``true'' Hitchin base via the Coulomb branch of certain 4D $\mathcal{N}=2$ supersymmetric quantum field theories.

We now define the \emph{residually nilpotent Hitchin systems}. Let $G$ be a complex reductive group with Lie algebra $\mathfrak{g}$. A \emph{nilpotent orbit} is an adjoint orbit $\mathbf{O}_e := G \cdot e$ for some nilpotent element $e \in \mathfrak{g}$. We incorporate nilpotent orbits into the framework of Hitchin systems over a pointed smooth algebraic curve $(\Sigma, x)$, where $x \in \Sigma$ is a fixed point:
\[
    h_{\Obar}: \mathbf{Higgs}_{\Obar} \longrightarrow \mathbf{H}_{\Obar}.
\]
The moduli space $\mathbf{Higgs}_{\Obar}$ is constructed via the Jacobson–Morozov resolution $G \times_{P_{JM}} \mathfrak{n}_2 \to \Obar$, where $P_{JM}$ and $\mathfrak{n}_2$ are defined from a standard $\mathfrak{sl}_2$-triple. See Section~\ref{sec:parHit_Nil} for further details. In type D case, the Hitchin base $\mathbf{H}_{\Obar}$ is defined as the space of coefficients of the characteristic polynomial of the Higgs fields (with the last term given by Pfaffian).

To illustrate how the geometry of residually nilpotent Hitchin systems connects to nilpotent orbits, we briefly recall some results from \cite{WWW24}.

By a result of Springer~\cite{Spr}, there exists an injection from the set of irreducible representations of the Weyl group into pairs $(e, \rho)$, where $e \in \mathfrak{g}$ is nilpotent and $\rho$ is an irreducible representation of the component group $A(\mathbf{O}_e) := G_e / G_e^\circ$. Equivalently, $\rho$ can be viewed as an equivariant local system on the orbit $\mathbf{O}_e$.

A nilpotent orbit $\mathbf{O}_e$ is said to be \emph{special} if the pair $(e, 1)$ corresponds to a special representation in the sense of Lusztig~\cite{Lus79}. Note that the generic nilpotent orbit associated to a parabolic subgroup---the Richardson orbit---is always special.

Since Langlands dual groups share the same Weyl group, there is a natural one-to-one correspondence between special nilpotent orbits in types B and C. We refer to this as the \emph{Springer duality}.

In \cite{WWW24}, we uncover deep connections among Springer theory, mirror symmetry, and Lusztig’s canonical quotient, via residually nilpotent Hitchin systems for $\Sp_{2n}/\SO_{2n+1}$. Geometrically, this manifests in dualities between generic Hitchin fibers and their images on both sides. As a result, we prove both Strominger--Yau--Zaslow and topological mirror symmetries for Langlands dual parabolic Hitchin systems of types B and C.

\subsection{Main results}

In this paper, we focus on type D. Unless otherwise specified, we take $G=\SO_{2n}$. Although $\SO_{2n}$ is Langlands self-dual, it exhibits unique behavior not present in other types. And our main results deal with generic Hitchin fibers and also Hitchin images. In particular, we show very nice geometric properties of residually nilpotent Hitchin systems associated with special nilpotent orbits. From the point view of Hitchin systems, they are also quite ``special".

We denote the residually nilpotent Hitchin system associate to $\Obar$ by $\mathbf{Higgs}_{\Obar}$. It has two connected componnets, and we use $\mathbf{Higgs}_{\Obar}^{\pm}$ to denote each one component.

Now we consider the corresponding Hitchin map:
\begin{equation} \label{intro.h_Obar}
    h_{\Obar}^{\pm}: \mathbf{Higgs}_{\Obar}^{\pm} \longrightarrow \mathbf{H}_{\Obar}.
\end{equation}
Here $\bf{H}_{\Obar}$ is the collection of coefficients of characteristic polynomials. As noticed in literature \cite{BK18, CDT, BDDP23}, $\bf{H}_{\Obar}$ can be a singular space even when $\bf{O}$ is Richardson. 

Tachikawa~\cite{Tac} proposes a physical conjecture that the Coulomb branch of certain 4D $\mathcal{N}=2$ supersymmetric quantum field theories—an affine space—can be realized as a finite cover of the Hitchin base in type D. Moreover, it is expected that the Hitchin map factors through this Coulomb branch.

Motivated by this insight, we define the \emph{Coulomb Hitchin base} as the affinization of $\mathbf{Higgs}_{\Obar}^{\pm}$:
\[
    \mathbf{A}_{\Obar} := \operatorname{Spec}\left( \mathbb{C} \left[ \mathbf{Higgs}_{\Obar}^{\pm} \right] \right).
\]

\begin{theorem}[Theorem~\ref{thm.O fiber} and Proposition~\ref{d o bar}]
     The residually nilpotent Hitchin map admits the following factorization:
    \begin{align*}
        \xymatrix{
            \mathbf{Higgs}^\pm_{\Obar} \ar[dr]_{h_{\Obar}^\pm} \ar[rr]^{h_{\Obar,\mathrm{aff}}^\pm} & & \mathbf{A}_{\Obar} \ar[dl]^{f_{\Obar}} \\
            & \mathbf{H}_{\Obar} &
        }
    \end{align*}
    such that 
    \begin{itemize}
        \item [(1)] $\bf{A}_{\Obar}$ is the normalization of $\bf{H}_{\Obar}$, hence $f_{\Obar}$ is generically one-to-one;
        \item [(2)] $\bf{A}_{\Obar}$ is isomorphic to product of an affine space with several affine Veronese varieties of degree $2$;
        \item [(3)] For generic $a$ in $\mathbf{A}_{\Obar}$ (or equivalently generic in $\mathbf{H}_{\Obar}$), the Hitchin fiber $(h_{\Obar, \mathrm{aff}}^\pm)^{-1}(a)$ is a torsor of an abelian variety $\Prym_{\Obar,a}$. 
    \end{itemize}
\end{theorem}

Moreover, we give a criterion for the smoothness of $\bf{A}_{\Obar}$ and $\bf{H}_{\Obar}$ in Corollary \ref{smoothness}. In particular, if they are smooth, then they are isomorphic to an affine space.

The abelian variety $\Prym_{\Obar,a}$ is constructed as a finite cover of the Prym variety of the normalized spectral curve, hence admits a natural polarization. Moreover, we have

\begin{theorem}[Theorem~\ref{thm.O fiber}]\label{thm.intro self-dual}
    The nilpotent orbit $\bf{O}$ is special if and only if $\Prym_{\Obar, a}$ is self-dual for generic $a$ in $\mathbf{A}_{\Obar}$.
\end{theorem}

Quoted from Lusztig~\cite{Lus97}, special orbits play a key role in several problems in representation theory. Unfortunately, their definition is total non-geometric, and as a result, special orbits are often regarded as rather mysterious objects. The theorem here strengthen the ``speciality" of special nilpotent orbits from the point view of Hitchin systems. This implies a possible geometric interpretation of special nilpotent orbits from Yun's global Springer theory. 

The self-duality of generic Hitchin fiber has more explanation from the perspective of mirror symmetry, which may provides connection between special nilpotent orbits and mirror symmetry. 

One may also compare this phenomenon to \cite[Theorem A]{WWW24}, which says that for two nilpotent orbits in type B and C, they are special and Springer dual if and only if the associated residually nilpotent Hitchin systems share isomorphic Hitchin base.

Let us now come to the usual parabolic Hitchin system. Let $P<G$ be a parabolic subgroup, recall the generalized Springer map:
\begin{equation*}
    \mu_P: T^*(G/P) \longrightarrow \Obar_R,
\end{equation*}
where $\mathbf{O}_R$ is the Richardson orbit (which is special), and $P$ is a polarization of $\mathbf{O}_R$.

We use $\bf{Higgs}_P$ to denote the moduli space of semistable parabolic Higgs bundles. This space also has two connected components and we use $\bf{Higgs}_P^{\pm}$ to denote either component.

It can be shown that $\mathbf{Higgs}^{\pm}_P$ and $\mathbf{Higgs}_{\Obar_R}^{\pm}$ share the same Hitchin base $\mathbf{H}_{\Obar_R}$. Now we also define the Coulomb Hitchin base as the affinization:
\[
    \mathbf{A}_{P} := \operatorname{Spec}\left( \mathbb{C} \left[ \mathbf{Higgs}_{P}^{\pm} \right] \right).
\] To understand the structure of $\mathbf{A}_P$, we begin by analyzing the generic Hitchin fiber of $h_{\Obar_R}$. 

Using the generalized Springer map, together with general results on the covers of Prym varieties established in \cite{WWW24}, we surprisely find that for generic $a$ in $\bf{H}_{\Obar_R}$, the fiber of $a$ in $\mathbf{Higgs}^{\pm}_P$ has $\deg \mu_P$ many components, since in the case of type B/C the group of connected components of generic parabolic Hitchin fibers is $\pi_1(\SO_{2n+1})$\footnote{In type B case, the moduli space has two connected components.} or $\pi_1(\Sp_{2n})$, which coincides with the non-parabolic setting.

Hence $\bf{A}_P$ should be a finite cover of $\bf{H}_{\Obar_R}$. In order to describe $\bf{A}_P$ explicitly, we need to construct new regular functions other than the coefficients of characteristic polynomial on the parabolic Higgs moduli $\mathbf{Higgs}^{\pm}_P$. We call these new regular functions by ``new Pfaffians" following Donagi, and the existence of which is inspired by \cite[Conjecture 1]{BDDP23}. See also \cite{DDD24}.

\begin{theorem}[Theorem~\ref{thm:Tac's conj}]
    The parabolic Hitchin map $h_{P}^{\pm}$ factors through the Coulomb Hitchin base:
    \begin{align*}
    \xymatrix{
        \bf{Higgs}^\pm_{P} \ar[dr]_{h_{P}^\pm} \ar[rr]^{h_{P,\rm{aff}}^\pm} &    & \bf{A}_{P} \ar[dl]^{f_P} \\
        & \bf{H_{\Obar_R}} &
    },
    \end{align*}
    such that 
    \begin{itemize}
        \item[1.] $\bf{A}_P$ is a finite cover of $\bf{H}_{\Obar_R}$ of degree $\deg \mu_P$;
        \item[2.] $\bf{A}_P$ is isomorphic to an affine space;
        \item[3.] For $a$ generic in $\bf{A}_P$, the fiber $(h_{P,\rm{aff}}^\pm)^{-1}(a)$ is a torsor of the self-dual abelian variety $\Prym_{\Obar_R, a}$.
    \end{itemize}
\end{theorem}

Roughly speaking, the construction of ``new Pfaffian" proceeds as follows. For a parabolic Higgs bundle $(E, \theta, (E|_x)_P)$, we firstly restrict it to a formal neighborhood of $x$, then by Theorem \ref{Thm. decomposition}, we can decompose $(E, \theta)$ locally as direct sum of some ``local Higgs bundles". The $\operatorname{SO_{2n}}$ structure on $E$ only induces $\operatorname{O_{2n_i}}$ structure on each direct summand and hence we can not define Pfaffian for each direct summand since we do not have a framing on each direct summand, see the discussion after Definition \ref{Coulomb Hitchin base}. 

However, by Proposition \ref{prop:maximal isotropic}, the $P$ reduction on $E|_x$ gives a choice of maximal isotropic subspace for the fiber at $x$ of some direct summands, which gives a framing on their fibers at $x$ by Lemma \ref{Volume form}. Consider global version of this construction, then we can define the ``new Pfaffian" using the framings. For details, see Section \ref{par Hitchin}.

Besides the construction of these ``new Pfaffian", one remarkable feature of this theorem is that the Coulomb Hitchin base $\bf{A}_P$ is always isomorphic to an affine space. This is conjectured by Tachikawa (\cite[Conjecture 8]{Tac}), although the parabolic Higgs moduli we considered here may be different. We believe that this feature deserves further explanation.

There are many analogies between Lagrangian fibrations of projective irreducible symplectic manifolds and Hitchin systems, e.g., the numerical P=W \cite{SY22} and P=W \cite{MS24,HMMS22,MSY25}. Concerning base, Matsushita, in a series of works \cite{Ma1, Ma2, Ma3}, proved that any holomorphic fibration from a projective irreducible sympletic manifold onto a projective manifold with \emph{connected fibers} is a Lagrangian fibration. Moreover, the base manifold is necessarily Fano and shares the same Betti numbers as the projective space. He further conjectured that the base must be the projective space, which is confirmed later by Hwang\cite{Hw08}. By analogy, in our setting the ``true’’ Hitchin base turns out to be an affine space.

\subsection*{Acknowledgements}
We would like to express our gratitude to Alexey Bondal for mentioning the work of Tachikawa~\cite{Tac} during the third-named author's visit to IPMU, which ultimately led to this work. Y. Wen would like to thank Hiraku Nakajima for his helpful discussions.

Part of this manuscript was written during all authors' visit to the Institute for Advanced Study in Mathematics at Zhejiang University. We sincerely appreciate the institute’s inspiring environment and support.

Bin Wang is supported by General Research Fund 14307022 and Early Career Scheme 24307121 of Michael McBreen. X. Wen is
supported by the Chongqing Natural Science Foundation Innovation and Development Joint Fund
(CSTB2023NSCQ-LZX0031). Y. Wen is supported by a KIAS Individual Grant (MG083902) at the Korea Institute for Advanced Study. 

\section{Preliminaries}

\subsection{Nilpotent Orbits}

We use ``$\equiv$'' to denote congruence modulo 2. Let $\mathcal{P}(2n)$ denote the set of partitions 
\[
\mathbf{d} = [d_1 \geq d_2 \geq \cdots \geq d_r]
\]
of $2n$. For a partition $\bf{d}$, we use $\#\bf{d}$ to denote the number of non-zero parts in $\bf{d}$. For $\varepsilon=\pm 1$, we define
\begin{align*}
    \mathcal{P}_\varepsilon(2n) = \left\{ [d_1, \ldots, d_r] \in \mathcal{P}(2n) \ \middle| \ \#\{ j \mid d_j = i \}\equiv 0 \text{ for all } i \text{ such that } (-1)^i = \varepsilon \right\}.
\end{align*}

It is well known that $\mathcal{P}_{1}(2n)$ is in bijection with the set of nilpotent orbits in $\mathfrak{so}_{2n}$, except in the case where all parts of a partition $\mathbf{d} \in \mathcal{P}_1(2n)$ are even. The corresponding orbit $\mathbf{O}_{\mathbf{d}}$ is then called \emph{very even}, and it splits into two distinct $\SO_{2n}$-orbits:
$\mathbf{O}_{\mathbf{d}} = \mathbf{O}_{\mathbf{d}}^{I} \bigsqcup \mathbf{O}_{\mathbf{d}}^{II}.$

\subsubsection{Special orbits}

By a result of Springer~\cite[Theorem 6.10]{Spr}, there is an injective map from the set of irreducible representations of the Weyl group to the set of irreducible $G$-equivariant local systems on nilpotent orbits. A nilpotent orbit $\mathbf{O}$ is called \emph{special} if it corresponds to the trivial local system and the associated Weyl group representation is special in the sense of Lusztig~\cite{Lus79}.

Equivalently, a partition $\mathbf{d} \in \mathcal{P}_1(2n)$ is called \emph{special} if its transpose partition $\mathbf{d}^t$ lies in $\mathcal{P}_{-1}(2n)$. Following~\cite{WWW24}, we classify partitions in type D into the following basic types:
\begin{itemize}
    \item Type D1: $[\alpha, \alpha]$, where $\alpha \equiv 1$.
    
    \item Type D1*: $[\beta, \beta]$, where $\beta \equiv 0$.
    
    \item Type D2: $[\alpha_1, \beta_1, \beta_1, \beta_2, \beta_2, \ldots, \beta_k, \beta_k, \alpha_2]$, where 
    \[
        \alpha_1 > \beta_1 \geq \beta_2 \geq \cdots \geq \beta_k > \alpha_2, \quad 
        \alpha_i \equiv 1, \quad 
        \beta_j \equiv 0, \quad 
        k \geq 0.
    \]
\end{itemize}

Any type D partition $\bf{d}$ can be uniquely decomposed as
\[
     \bf{d} = [\bf{T}_1, \bf{T}_2, \ldots, \bf{T}_l]
\]
where each block $\bf{T}_i$ is of types D1, D1*, or D2. See \cite{FWW} for further detail on classical types.

\begin{lemma}\label{lem.special}
	A type D partition $\bf{d}$ is special if and only if it does not contain a type D2 block of the form $[\alpha_1, \beta_1, \beta_1, \ldots, \beta_k, \beta_k, \alpha_2]$ with $k \geq 1$.
\end{lemma}

\subsubsection{Richardson orbits}\label{sec:Richarson}
 
A nilpotent orbit is called \emph{Richardson}, denoted by $\bf{O}_{R}$, if there exists a parabolic subgroup $P<G$ such the image of moment map (also known as \emph{generalized Springer map}) is the closure of the orbit:
\begin{align}\label{Springer map}
    \mu_P: T^*(\SO_{2n}/P) \longrightarrow \Obar_R.
\end{align}
This parabolic subgroup is called a \emph{polarization} of the Richardson orbit.

In general, this map is generically finite. Its degree can be computed using the structure of $P$. Let $\gp = \mathrm{Lie}(P)$, which admits a Levi decomposition $\gp = \gl \oplus \gn$. The Levi subalgebra $\gl$ has the form:
\begin{align} \label{Levi type}
    \gl \cong \mathfrak{gl}_{p_1} \oplus \ldots \oplus \mathfrak{gl}_{p_m} \oplus \mathfrak{so}_{q}, 
\end{align}
where $p_i \geq 0$, $ q \geq 0$ even, and $q \neq 2$. It is known that when $q = 0$, there exist two non-conjugate parabolic subgroups of $\SO_{2n}$ sharing the same Levi type as in~\eqref{Levi type}.

For each $i=1,\ldots,m$, define the partition $\bf{1}_i= \left[ 1_{i1}, \ldots, 1_{i p_i} \right]$ (all ones) of $p_i$, and let $\bf{1}_0= \left[ 1_{01}, \ldots, 1_{0q} \right]$ be a partition in $\mathfrak{so}_{q}$. 

We then define a partition $\bf{d}_P \in \cP(2n)$ associated with $P$ as:
\begin{align}
    \bf{d}_P &= \left[ d'_1, \ldots, d'_r \right],	\quad \text{with} \quad d'_j = 2 \sum_{i=1}^m 1_{ij} + 1_{0j}. \label{induction}
\end{align}
If $\mathbf{d}_P  \notin \mathcal{P}_1(2n)$, then at least one even part appears with an odd multiplicity. Let k be the largest such part. We define the \emph{D-collapse} of $\mathbf{d}_P$, denoted $(\mathbf{d}_P)_D$, by iteratively:
\begin{enumerate}
	\item[1.] Replacing the last occurrence of $k$ by $k-1$,
	\item[2.] Replacing the first subsequent part  $k'< k-1$ by $k'+1$,
\end{enumerate}
until the resulting partition lies in $\cP_1(2n)$.

Geometrically, if $\bf{O}_{\bf{d}_P}$ is the nilpotent orbit in $\mathfrak{sl}_{2n}$ associated with $\bf{d}_P$, then $\Obar_{(\mathbf{d}_P)_D} = \Obar_{\bf{d}_P} \cap \mathfrak{so}_{2n} $.

The degree of the generalized Springer map is closely tied to this collapse. Define
\begin{align}\label{I(P)}
    I(P) = \{ j \in \N \mid j \equiv 1,\; d'_j \equiv 0 \; \text{and}\; d'_j \geq d'_{j+1}+2 \}.
\end{align}

\begin{lemma}\label{lem:stru of Ti}
Let $\mathbf{d} = [d_1, \ldots, d_r] = [\mathbf{T}_1, \mathbf{T}_2, \ldots, \mathbf{T}_l] \in \cP_1(2n)$ be a partition associated to Richardson orbit $\bf{O}_R$. Let $P<G$ be a polarization as above. Then it satisfies the following conditions:

\begin{itemize}
    \item[1.] It does not contain a type D2 block of the form 
    \[
        [\alpha_1, \beta_1, \beta_1, \ldots, \beta_k, \beta_k, \alpha_2] \quad \text{with } k \geq 1.
    \]
    Consequently, each $\mathbf{T}_i$ consists of two parts, and the total number of parts is even, i.e., $r = 2l$.

    \item[2.] There exists an integer $m_0$ such that:
    \begin{itemize}
        \item[2.1.] For $i < m_0$, each block $\mathbf{T}_i$ is either of type D1 or D2;
        \item[2.2.] The block $\mathbf{T}_{m_0}$ is of type D1*;
        \item[2.3.] For $i> m_0$, if $\bf{T}_i=[d_{2i-1},d_{2i}]$ is not of type D1*, then $2i-1 \in I(P)$;
        \item[2.4.] For any $\mathbf{T}_i = [d_{2i-1}, d_{2i}]$ and $\mathbf{T}_{j} = [d_{2j-1}, d_{2j}]$, $j>i$, if $2i-1 \in I(P)$, then $\bf{T}_{i-1} \neq \bf{T}_i$ and $d_{2i} > d_{2j-1}$. Moreover, if $d_{2j-1} \equiv 1$, then $2j-1 \in I(P)$.  
    \end{itemize}
\end{itemize}
\end{lemma}

\begin{proof}
    It follows from the rule of D-collapse and the definition of $I(P)$.
\end{proof}

Using the following Jordan basis, by the work of Hesselink~\cite{He78}, we are able to describe the generic fiber of the generalized Springer map~\eqref{Springer map}.

Let the partition $\mathbf{d} = [d_1,\ldots,d_r]$ define a Young tableau $\mathrm{Y}(\mathbf{d}) \subset \mathbb{Z}_{>0}^2$, where $(i, j) \in \mathrm{Y}(\mathbf{d})$ if and only if $1 \leq j \leq r$ and $1 \leq i \leq d_j$.

We choose a Jordan basis of $\mathbb{C}^{2n}$, denoted by $\left\{ v(i,j) \right\}_{(i,j) \in \mathrm{Y}(\mathbf{d})}$, for a nilpotent element $e \in \mathbf{O}$ as follows:
\begin{itemize}
    \item $e \cdot v(i,j) = v(i-1, j)$ for $i > 1$, and $e \cdot v(1,j) = 0$;
    \item The pairing satisfies $\langle v(i,j), v(p,q) \rangle \neq 0$ if and only if $i + p = d_j + 1$ and $q = \tau(j)$. Here $\langle -, - \rangle$ is the bilinear form on $\mathbb{C}^{2n}$, and $\tau$ is an involutive permutation of $\{1, \ldots, r\}$ such that $\tau^2 = \mathrm{id}$, $d_{\tau(j)} = d_j$, and $\tau(j) \neq j$ if $d_j \equiv 0 \pmod{2}$. We fix a choice of $\tau$ such that $\tau(j) = j$ whenever $d_j$ is odd.
\end{itemize}

Let $\mathbf{O}_R$ be a Richardson orbit with polarization $P$, and suppose the Levi subalgebra is given by~\eqref{Levi type}. Define
\[
    V_{i} = \mathbb{C}v \left( \frac{ d_{2i-1}+1}{2}, 2i-1 \right) \oplus \mathbb{C} v \left( \frac{d_{2i}+1}{2}, 2i \right), \quad \text{for} \quad 2i-1 \in I(P).
\]
Then from \cite[Theorem 7.1]{He78} or \cite{FWW}, we have the following description of generic fibers of \eqref{Springer map}.

\begin{proposition} \label{prop.deg_mu_P}
The generic fiber of the generalized Springer map is 
\begin{align}\label{Spal_fib}
    \prod_{2i-1 \in I(P)} \OG(1, V_{i}), \quad \text{for}\quad q \geq 4.
\end{align}
When $q=0$, \eqref{Spal_fib} is the union of the generic fibers associated to two non-conjugate parabolic subgroups with the same Levi type.

As a consequence, the degree of the map $\mu_P$ is given by
\[
    \deg \mu_P = 
    \begin{cases}
        2^{\# I(P)}, & \text{if } q \geq 4, \\
        2^{\# I(P) - 1}, & \text{if } q = 0.
    \end{cases}
\]
\end{proposition}

\begin{proof}
To study the generic fiber of generalized Springer map
\[
    T^*(\SO_{2n}/P) \longrightarrow \Obar_R,
\]
it suffices to construct admissible filtrations $\Fil^{\bullet}_{P}$ of the form
\begin{align*}
    \CC^{2n}=F_0 \supset F_1 \supset \ldots \supset F_{m} \supset F_{m}^\perp \supset \ldots \supset F_{1}^\perp \supset F_{0}^\perp = 0,
\end{align*}
of Levi type $(p_1,\ldots,p_m;q)$, i.e.,
\begin{align*}
    \dim F_m/F_m^\perp = q, \; ( q \equiv 0,\; q \neq 2) \quad \dim F_{i-1}/F_i = p_i \quad \text{for} \quad i=1, \ldots, m,  
\end{align*}
such that the fixed nilpotent element $e \in \bf{O}_R$ satisfies
\[
\begin{cases}
    e(F_i) \subset F_{i+1}, \quad \text{for} \quad 0 \leq i \leq m-1, \\
    e(F_m) \subset (F_m)^\perp, \\
    e(F_{m-i}^\perp) \subset F_{m-i-1}^\perp, \quad \text{for} \quad 0 \leq i \leq m-1.
\end{cases}
\]
To simplify the argument, we assume $p_1 \leq p_2 \leq \ldots \leq p_m$.

Let the partition associated to $\mathbf{O}_R$ be $\mathbf{d} = [d_1, d_2, \ldots]$, and let the induced partition from $P$ be $\mathbf{d}_P = [d'_1, d'_2, \ldots]$, defined as in \eqref{induction}. Then $\mathbf{d} = (\mathbf{d}_P)_D$, the D-collapse of $\mathbf{d}_P$, and the collapses occur at indices $j$ such that $d'_j \equiv d'_{j+1} \equiv 0$ and $d'_j > d'_{j+1}$, i.e., for $j \in I(P)$. From the construction \eqref{induction}, we know that for $p_i \equiv 1$ such that $q < p_i \neq p_{i+1}$, then $p_i \in I(P)$.

Define the subspace
\[
    E=\bigoplus\limits_{ \substack{1 \leq i \leq m \,:\, p_i \notin I(P) \\ 1 \leq j \leq p_i}  } \CC v(m+1-i, j)  \oplus \bigoplus\limits_{ \substack{1 \leq i \leq m \,:\, p_i \in I(P) \\ 1 \leq j \leq p_i-1}  }  \CC v(m+1-i, j),
\]
since $F_m^\perp \subset F_m$, it can be shown that
\[
    E
    \subset F_m^\perp \subset 
    E \oplus \bigoplus\limits_{i \,:\, 2i-1 \in I(P)} V_i.
\]
Note that $\dim F_m^\perp = \sum_{i=1}^m p_i$. Moreover, due to the induced non-degenerate pairing on each $V_i$, it remains to choose a maximal isotropic subspace
\[
    W \subset \bigoplus\limits_{i \,:\, 2i-1 \in I(P)} V_i.
\]

Let
\begin{align*}
E_i= \frac{\ker \left( e^{m+1-i} \right) \cap F^\perp_m}{\ker \left( e^{m-i} \right) \cap F^\perp_m}.
\end{align*} 
Since $e^{i-1}(F_m^\perp) \subset F_{m+1-i}^\perp$, it follows that $\dim E_i \leq p_i$. Given that $\dim F_m^\perp = \sum_{i=1}^m p_i$, we must have equality, i.e., $\dim E_i = p_i$. From this, we deduce that
\[
    W = \bigoplus_{i \,:\, 2i-1 \in I(P)} \OG(1, V_i).
\]

Moreover, by dimension counting we obtain $e^k(F^\perp_m) = F^\perp_{m-k}$. Hence, to construct $\Fil_P^\bullet$, it suffices to determine $F_m^\perp$. This completes the construction.

Finally, we note that when $q = 0$, there exist two non-conjugate parabolic subgroups with the same Levi type. In this case, the disconnected fiber constructed above is the union of the generic fibers associated with these non-conjugate parabolics.
\end{proof}

\subsection{Kazhdan--Lusztig Map}

Let $\sN \subset \mathfrak{g}$ denote the set of nilpotent elements, which decomposes into a finite union of nilpotent $G$-orbits under the adjoint action. We denote a typical nilpotent orbit by $\mathbf{O}$. The Kazhdan--Lusztig map is defined as follows.

We use $\mathcal{O}$ to denote the formal power series ring $\mathbb{C}[\![t]\!]$ and use $\mathcal{K}$ to denote its fractional field. Given an element $\theta_0 \in \mathbf{O}$, choose a generic lift $\theta \in \theta_0 + tL^+\g$ such that $\theta$ is regular semisimple in the loop algebra $L\mathfrak{g}$. Then its centralizer $Z_{L G}(\theta)$ is a maximal torus in the loop group $L G$. Kazhdan and Lusztig \cite{KL88} showed that the conjugacy class of $Z_{L G}(\theta)$ is independent of the choice of lift, yielding a well-defined map:
\begin{align*}
\{\text{Nilpotent orbits}\}&\rightarrow \{\text{rational conjugacy classes of maximal tori in } L G\} \\
\theta_0 &\mapsto Z_{L G}(\theta)
\end{align*}

Moreover, they proved that rational conjugacy classes of maximal tori in $L G$ are in bijection with conjugacy classes in the Weyl group $\mathsf{W}$. Thus, we obtain the composite Kazhdan--Lusztig map:
\[
\{\text{Nilpotent orbits}\} \xrightarrow{\mathrm{KL}} \{\text{conjugacy classes in the Weyl group } \mathsf{W}\}.
\]

Yun~\cite{Yun21} proved that the Kazhdan--Lusztig map is injective. For classical groups, a concrete description was given by Spaltenstein~\cite{Spa88}, which we now recall. For more details, see Spaltenstein~\cite{Spa88}.

\subsubsection{Spaltenstein's Interpretation}

Let $\theta_0 \in \mathbf{O}$ be a nilpotent element with partition $\mathbf{d} = [d_1, \cdots, d_r]$. 
\begin{definition}
    We denote $L^+\g_{\bf{O}}:=\{\theta\in L^+\g\mid\theta\mod t\in\bf{O}\}$. We denote $L^+\gc_{\bf{O}}$ as the image of the Chevalley map in $L^+\gc$.
\end{definition}
Yun~\cite{Yun21} shows that $L^+\gc_\bf{O}$ is a finitely presented constructible locally closed subscheme of $L^+\gc$. Now choose a generic lift $\theta = \theta_0 + \cdots \in L^+\g$. In other words, $\theta$ is a generic elements in $L^+\g_{\bf{O}}$. Here genericity is as defined in \cite{KL88}, or see generic ``shallow elements" in \cite{Yun21}. Consider the characteristic polynomial
\[
    \chi(\theta) = \det(\lambda I - \theta) \in \mathbb{C}[\![t]\!][\lambda].
\]
If $G = \SO_{2n}$, we set $f(\lambda) = \chi(\theta)$ and factor it into irreducibles:
\begin{equation}\label{eq:decomposition of char}
f = \prod_{i=1}^r f_i, \quad f_i \in \mathbb{C}[\![t]\!][\lambda].
\end{equation}
Let $e_i = \deg f_i$ and define the partition $\mathbf{deg} = [e_1, \cdots ,e_r]$. Notice that the length of $\bf{d}$ is equal to the length of $\bf{deg}$. Since $f(-\lambda) = f(\lambda)$, each irreducible factor fits into one of the following cases, which determines how we construct the pair of partitions $(\alpha, \beta)$:

\begin{enumerate}
    \item[1.] If $e_i$ is odd, then there exists a unique $i'$ such that $f_i(-\lambda) = -f_{i'}(\lambda)$. In this case, $\alpha$ gets a part $e_i$.
    
    \item[2.] If $e_i$ is even and $f_i(\lambda) = f_i(-\lambda)$, then $\beta$ gets a part $\frac{e_i}{2}$.
    
    \item[3.] If $e_i$ is even and there exists a unique $i'$ such that $f_i(-\lambda) = f_{i'}(\lambda)$, then $\alpha$ gets a part $e_i$. We refer to such factors as \emph{even dual to each other}.
\end{enumerate}

\begin{remark}\label{rmk:dual to each other}
In other words, parts in $\alpha$ correspond to pairs of irreducible factors that are dual to each other.
\end{remark}
For later use, we state the decomposition for a non-algebraically closed field $\Bbbk$ where $\mathbb{C}\subset\Bbbk\subset\bar{\Bbbk}$. Given $\theta_0\in \bf{O}(\Bbbk)$. We choose a generic lifting $\theta=\theta_0+\ldots\in L^+\g(\Bbbk)$. Let $f$ be as above. In particular, $f\in \Bbbk[\![t]\!][\lambda]$. We still have the decomposition as in \eqref{eq:decomposition of char}, with each $f_i\in\bar{\Bbbk}[\![t]\!][\lambda]$. 
\begin{lemma}(\cite{SWW22},\cite[\S 6.2]{She23}\label{lem:factorization over non-closed field})
    For each $d_i\in \bf{deg}$, the product $f_{d_i}=\prod\limits_{\{j \,\mid\, \deg f_j=d_i\}}f_j$ has coefficients in $\Bbbk[\![t]\!]$.
\end{lemma}
The basic idea is that each irreducible factor of degree $d_i$ define a smooth irreducible branch in the $d_i$-th blow-up over the marked points. When over an algebraically field, they are separated, in the sense that they have $\#\{i\mid\deg f_i=d_i\}$ intersection points with the exceptional divisor. The intersection is defined over the non-algebraically closed field $\Bbbk$ though not separated into distinct points. This is why we use the product of irreducible factors of same degrees.

\begin{proposition}[{\cite[Proposition 6.4]{Spa88}}] \label{prop:deg}
    Given a partition $\bf{d}=[d_1, \cdots , d_r]=[\bf{T}_1, \cdots , \bf{T}_l]$, the associated partition $\bf{deg}$ is given by $[\bf{T}_1', \cdots , \bf{T}_l']$, where $\bf{T}_i'$ is defined as follows:
    \begin{itemize}
        \item If $\bf{T}_i$ is of type D1 or D1*, then $\bf{T}_i'=\bf{T}_i$;
        \item If $\bf{T}_i=[\alpha_1, \beta_1, \beta_1, \cdots , \alpha_2]$ is of type D2, then $\bf{T}_i'=[\alpha_1-1, \beta_1, \beta_1, \cdots , \alpha_2+1]$.
    \end{itemize}
    Moreover, the Kazhdan--Lusztig map is given by the following rule: if $\bf{T}_i$ is of type D1 or D1*, for $e\in \bf{T}_i'$, then $\alpha$ gets a part $e$; and if $\bf{T}_i$ is of type D2, for $e\in \bf{T}_i'$, then $\beta$ gets a part $\frac{e}{2}$.
\end{proposition}

\section{Decompose Local Higgs Bundles}

 For $\theta: \sK^{2n}\rightarrow \sK^{2n}$, we fix a nondegenerate symmetric two-form $g:\sK^{2n}\otimes \sK^{2n}\rightarrow \sK$, such that $g(\theta -,-)+g(-,\theta -)=0$. The set of residually nilpotent $\SO_{2n}$-Higgs bundle associated with $\theta$ and a nilpotent orbit $\bf{O}$ is given by

\[
{\mathbf{Gr}}_{\theta, \bf{O}}:=\left\{E \subset \sK^{2n}\ \left| \
\begin{aligned}
    & E \;\text{is a rank-}2n\; \text{lattice, such that} \\
    & g|_{E} \; \text{is perfect, with values in}\  \sO; \\
    & E \text{ is } \theta \text{ invariant}; \\
    & \theta(0) \in \bf{O}.
\end{aligned}\right.\right\}.
\]

\begin{definition}
    For $\bf{d}$ and $\bf{deg}$ as before, we define $\tilde{\beta}(\bf{d})=\#\{e_i\in \bf{deg}\mid e_i\equiv 0\}$,  $c(\bf{d})=\#\{\text{type D2}\}$ and  $$\beta(\bf{d})=\tilde{\beta}(\bf{d})-2\#\{\text{type D1*}\}.$$
\end{definition}

\begin{theorem}\label{Thm. decomposition}
    If $\chi(\theta)=f(\lambda)$ is generic in $L\mathfrak{c}_{\bf{O}}$, there is a finite cover of degree $2^{\beta(\bf{d})-c(\bf{d})}$: 
    \[
    \mathbf{Gr}_{\theta, \mathbf{O}} \longrightarrow     
       \{ (\sL\hookrightarrow \sK_f, \sigma^*\sL\cong\sL^{\vee}\otimes\sO(\sR) \mid \sL \text{ is a  rank $1$ free module of} \;\bar{\sO}_f
       \},
    \]
    where $\sigma$ is the involution $\lambda\mapsto -\lambda$ and $\sR$ is the divisor corresponding to fixed points.
\end{theorem}

In the following, we fix $E\in \bf{Gr}_{\theta, \bf{O}}$ where $\chi(\theta)$ is generic in $L\mathfrak{c}_{\bf{O}}$.

\begin{definition}
    For a submodule $i_F: F\hookrightarrow E$, we say that $F$ is a \emph{$\theta$ direct summand} of $E$ if
    \begin{itemize}
        \item $F$ is $\theta$ invariant.
        \item There is an $\mathcal{O}$ morphism $s_F: E \rightarrow F $ such that it is compatible with $\theta$ and $s_F\circ i_F$ is identity.
    \end{itemize}  
\end{definition}

We have the decomposition of characteristic polynomial:
\[
\chi(\theta)=f(\lambda)=\prod_{i=1}^{k}f_i.
\]
Since $f(-\lambda)=f(\lambda)$, we can talk about $f_i$ are selfdual or dual to another factor. The following proposition is crucial for a moduli interpretation of generic Hitchin fibers. We use $K_i$ to denote $\Ker f_{i}(\theta)\subset E$.
\begin{proposition}\label{prop:direct summand even dual}
    Consider a pair of irreducible factors $(f_i,f_{i'})$ such that $f_i(\lambda)=f_{i'}(-\lambda)$ and $\deg f_i$ is even. Then $K_i\oplus K_{i'}$ is a $\theta$ direct summand of $E$.
\end{proposition}
\begin{proof}
    Let $E_i$ be the saturation of $K_i\oplus K_{i'}$ in $E$. Then we have:
    \[
    E=E_i\oplus E_i^{\perp},
    \]
    and each is stable under $\theta$. Notice that $g|_{E_i}$ is nondegenerate. Replacing $E$ by $E_i$, we reduce to the following situation: $f=f_1f_2$ such that $f_1(\lambda)=f_2(-\lambda)$ and they have even degrees $2m$. Then we need to show that $K_1, K_2$ are $\theta$ direct summands of $E$. 

    Since $f=f_1f_2$ with $\deg f_1=\deg f_2=2m$, the partition of $\theta(0)$ has to be $\ge [2m,2m]$. We choose $v\in E(0)$ such that $$K(0)=\langle v,\theta(0)v,\ldots,\theta(0)^{2m-1}v \rangle .$$
    Suppose that the largest term in the partition is $d_1>2m$ which must be odd. By Jacobson--Morozov, we can find a subspace of the following form:
    \[
    W(0):=\langle w,\theta(0)w,\ldots,\theta(0)^{d_1-1}w\rangle
    \]
    and  $g$ is nondegenerate when restricted to $W(0)$.  we can rewrite $E(0)=W(0)\oplus W(0)^{\perp}$. Since $4m-d_1<2m$, $v$ is nonzero under the projection $E(0)\rightarrow W(0)$. In particular, $K_1(0)$ cannot be isotropic, a contradiction.
   
   Hence we know that $\theta(0)$ has the partition $[2m,2m]$. Then the proposition follows from Lemma \ref{lem:zero intersection}.
\end{proof}

\begin{lemma}\label{lem:even intersection}
   $\codim \left( K_1(0)\cap K_2(0) \right)\equiv 0$.
\end{lemma}
 
\begin{proof}
    Consider the following Jordan basis \[\{w_1,\theta(0)w_1,\ldots, \theta(0)^{2m-1}w_1, w_2,\theta(0)w_2,\ldots, \theta(0)^{2m-1}w_2\}.\] We choose a generator of $K_1(0)$ by:
    \[
    v=\sum_{i=0}^{2m-1} (a_{i1}\theta(0)^{i}w_1+a_{i2}\theta(0)^{i}w_2)=\sum_{i=0}^{2m-1}v_i
    \]
    where $v_i=a_{i1}\theta(0)^{i}w_1+a_{i2}\theta(0)^{i}w_2$ (i.e. decompose as weight spaces).

    Since $K_1(0)$ is maximal isotropic, we have the following equations:
    \[ 
    \langle v, \theta(0)^{2\ell}v\rangle=0,\ell=0,1,\ldots, m-1.
    \]
    From the equation:
    \[
    \langle v, \theta(0)^{2m-2}v\rangle=0
    \]
    we have that:
    \[
    \langle v_1,\theta(0)^{2m-2}v_2\rangle=\langle v_2,\theta(0)^{2m-2}v_1\rangle=0
    \]
    Hence there exists $\lambda\in\mathbb{C}$ such that:
    \[
    v_2=\lambda \theta(0) v_1
    \]
    Then we can choose a generator of $K_1(0)$ with $v_2=0$ by $v\mapsto v-\lambda \theta(0)v$. And by induction, we can make all the $v_{2i}=0$ in the expression of a generator of $K_1(0)$. As a result, we can find generators for $K_1(0)$ of the form $v=v_1+v_3+\ldots+v_{2m-1}$ where $v_i\in \langle \theta(0)^{i-1}w_1,\theta(0)^{i-1}w_2\rangle$. The same works for $K_2(0)$. 

    If the generators $v,v'$ of $K_1(0),K_2(0)$ share same $v_1$ but essentially different $v_3$, $\codim K_1(0)\cap K_2(0)=2$. And so on.
\end{proof}
\begin{remark}
    The above argument does not hold for type C. For example, for $[2,2]$ in type C, dual to each other cases, you can choose $v=w_1+a\theta(0)w_2, v'=w_1+b\theta(0)w_2$. Notice that they all satisfy that $\langle v, \theta(0)v\rangle=0$, i.e., maximal isotropic. But $\dim K_1(0)\cap K_2(0)=1$.
\end{remark}
The following lemma is a complement to Lemma \ref{lem:even intersection}.
\begin{lemma}\label{lem:zero intersection}
    We have $K_1(0)\cap K_2(0)=0$.
\end{lemma}
\begin{proof}
    For notation ease, we prove for $m=2$ and argue by contradiction. Suppose not, then for $v(0)\in K_1(0)\cap K_2(0)$, we choose two liftings in $K_1,K_2$ as $v=v(0)+t\tilde{v}, v'=v(0)+t\tilde{v'}$. We write:
    \begin{align*}
        f_1(\lambda)=\lambda^4+a\lambda^3+b\lambda^2+c\lambda+d\\
        f_2(\lambda)=\lambda^4-a\lambda^3+b\lambda^2-c\lambda+d
    \end{align*}
    Then we have:
    \begin{align}\label{eq:generators}
        \frac{\theta^{4}v}{t}|_{t=0}=-\frac{a}{t}|_{t=0}\theta(0)^3v(0)-\frac{b}{t}|_{t=0}\theta(0)^2v(0)-\frac{c}{t}|_{t=0}\theta(0)v(0)-\frac{d}{t}|_{t=0}v(0)\\
        \frac{\theta^{4}v'}{t}|_{t=0}=\frac{a}{t}|_{t=0}\theta(0)^3v'(0)-\frac{b}{t}|_{t=0}\theta(0)^2v'(0)+\frac{c}{t}|_{t=0}\theta(0)v'(0)-\frac{d}{t}|_{t=0}v'(0)
    \end{align}
    We may put $\theta=\theta(0)+t\tilde{\theta}$. Notice that:
    \[
    \theta^4 v=t((\theta(0)^{3}\tilde{\theta}|_{t=0}+\theta(0)^{2}\tilde{\theta}|_{t=0}\theta(0)+\cdots)v(0))+\cdots
    \] 
    Hence We have
    \begin{align*}
    &\frac{\theta^{4}v}{t}|_{t=0}-\frac{\theta^{4}v'}{t}|_{t=0}\\
    =&(\theta(0)^{3}\tilde{\theta}|_{t=0}+\theta(0)^{2}\tilde{\theta}|_{t=0}\theta(0)+\cdots)(v(0)-v(0))\\   
    =&0
    \end{align*}
    
    Now the difference of RHS of Equations \eqref{eq:generators} is:
    \[
    2\big(\frac{a}{t}|_{t=0}\theta(0)^3v(0)+\frac{c}{t}|_{t=0}\theta(0)v(0)\big)
    \] 
    which should be 0 then. By the genericity of $f_1,f_2$, $\theta(0)v(0)=0$. Hence $K_1(0)\cap K_2(0)\subset K_1(0)\cap \ker\theta(0)$. Notice that $\dim K_1(0)\cap \ker\theta(0)=1$, combined with Lemma \ref{lem:even intersection}, $K_1(0)\cap K_2(0)$ has to be 0.
\end{proof}

\begin{proof}[Proof of Theorem \ref{Thm. decomposition}]
    For $E\in \bf{Gr}_{\theta, \bf{O}}$, since $\chi(\theta)$ is generic, then by Proposition \ref{prop:direct summand even dual}, we have the following $\theta$ invariant decomposition $$E\cong \big(\bigoplus_{d_j\in \bf{T}_i \text{\ which\ is\ of\ D1*} }K_j\big)\bigoplus E_{12}.$$

    We have an induced non-degenerate symmetric pairing $g_{12}$ and an induced Higgs field $\theta_{12}$ on $E_{12}$ so that $g_{12}(\theta_{12} -,-)+g_{12}(-,\theta_{12} -)=0$. Notice that $\chi(\theta_{12})$ is also generic and by Proposition \ref{prop:deg}, for irreducible factors $f_{j}(\lambda)$ and $f_{j^\prime}(\lambda)$ of $\chi(\theta_{12})$, if $f_j(\lambda)=\pm f_{j^\prime}(-\lambda)$, then $\deg f_j(\lambda)=f_{j^\prime}(\lambda)$ is odd. Moreover, the partition of $\theta_{12}(0)$ contains no blocks of type D1*.

    This case has been considered in \cite{WWW24}, here we recall the idea of the proof.

    Firstly, by \cite[Theorem 2.8]{WWW24}, we have a $\theta$ invariant decomposition $$E_{12}\cong E_1 \bigoplus E_2$$ where $E_{1}$ is the direct sum of $K_j$ such that $d_j\in \bf{T}_i$ which is of type D1. Then we only need to consider $E_2$.

    Secondly, by  \cite[Corollary 2.41]{WWW24}, we have the following $\theta$ invariant decomposition $$E_{2}\cong \bigoplus_{ \bf{T}_j \text{ of type D2} } E_{\bf{T}_i}$$ here we define $T_i(\lambda)=\prod_{d_j\in \bf{T}_i}f_{j}(\lambda)$, and then define $E_{\bf{T}_i}$ to be $\Ker T_{i}(\theta)$. Thus we reduce to the case to consider $E_{\bf{T}_i}$ for each $\bf{T}_i$ of type D2, which is assumed to be $\bf{T}_i=[\alpha_1, \beta_1, \beta_1, \cdots \beta_k, \beta_k, \alpha_2]$.

    Finally, we consider the following exact sequence: $$0\longrightarrow \bigoplus_{d_j\in \bf{T}_i}K_j\longrightarrow E_{\bf{T}_i}\longrightarrow W_i\longrightarrow 0$$ where $W_i$ is a vector space over $\mathbb{C}$ of dimension $k+1=\#\bf{T}_i/2$, since by \cite[Proposition 2.29]{WWW24}, the restriction of the pairing $g$ on this $K_j$ is ``$1$-degenerate", which means that the induced morphism $K_j\rightarrow K_j^\vee$ makes $K_j$ to be a subsheaf of $K_j^\vee$ of codimension $1$.

    In summary, by sending $E$ to $\oplus_{d_j\in \bf{d}}K_j$, we defined our morphism in the theorem.

    To see the degree of this morphism, we start with a module $\oplus_{d_j\in \bf{d}}K_j$ as above and then try to find how many $E$ we can construct from such module. By the construction above, we only need to see how to construct a $E_{\bf{T}_i}$ from $\oplus_{d_j\in \bf{T}_i}K_j$ for a $\bf{T}_i$ of type D2. 

    Now this is characterized as follows: we have $$0\longrightarrow \bigoplus_{d_j\in \bf{T}_i}K_j\longrightarrow \bigoplus_{d_j\in \bf{T}_i}K_j^\vee \longrightarrow Q_i\longrightarrow 0$$ where $Q_i$ is a vector space of dimension $2k+2=\#\bf{T}_i$ over $\mathbb{C}$ since the restriction of the pairing $g$ on each $K_j$ is ``$1$-degenerate". Notice that there is an induced nondegenerate pairing on $Q_i$.

    To reconstruct a $E_{\bf{T}_i}$, it is suffice to find a particular kind of subspaces $W$ of $Q_i$, since $E_{\bf{T}_i}$ should contain $\oplus_{d_j\in \bf{T}_i}K_j$ and should be a submodule of $\oplus_{d_j\in \bf{T}_i}K_j^\vee$. Since we require that there is a non-degenerate pairing on $E_{\bf{T}_i}$, such $W$ must be maximal isotropic under the induced pairing on $Q_i$. The other constraint on $W$ is that the morphism $\theta_{\bf{T}_i}: E_{\bf{T}_i}\rightarrow E_{\bf{T}_i}$ must have partition $\bf{T}_i$ when $t=0$. We call such $W$ to be a $\iota$-isotropic subspace of $Q_i$. In \cite[Proposition 2.43]{WWW24}, we show that there are exactly $2^{\#\bf{T}_i-1}$ many $\iota$-isotropic subspaces of $Q_i$ and describe each $\iota$-isotropic subspace explicitly.

    Thus the degree of the morphism we constructed is $2^{\beta(\bf{d})-c(\bf{d})}$.
\end{proof}

Although the following corollary is already indicated by \cite[Corollary 2.41]{WWW24}, we state here for later use.

\begin{corollary}\label{coro:decomposition over general field}
    Consider $E\in \bf{Gr}_{\theta, \bf{O}}$ as in the Theorem \ref{Thm. decomposition} above. As before, we write the partition of $\bf{O}$ by $\bf{d}=[d_1, \cdots , d_r]=[\bf{T}_1, \cdots , \bf{T}_l]$. We have a decomposition $$E\cong \bigoplus_{i=1}^l\Ker T_{i}(\theta).$$ Each direct summand has an induced nondegenerate pairing. Moreover, if $T_{i}(\lambda)$ is a factor of $\chi(\theta)$ over a general field, then $\Ker T_{i}(\theta)$ is also a $\theta$ direct summand of $E$ over this general field.
\end{corollary}

\section{Residually Nilpotent Hitchin Systems}\label{sec:parHit_Nil}

Consider a nilpotent element $e \in \mathfrak{g} = \mathrm{Lie}(G)$. Jacobson--Morozov theorem ensures the existence a triple $\{ e, h, f \} \cong \mathfrak{sl}_2$. The adjoint action of $h$ induces a grading $\mathfrak{g} = \bigoplus_{i \in \mathbb{Z}} \mathfrak{g}_i.$
Define the parabolic subalgebra $\mathfrak{p}_{JM} := \bigoplus_{i \geq 0} \mathfrak{g}_i$ and the subspace $\mathfrak{n}_2 := \bigoplus_{i \geq 2} \mathfrak{g}_i$. Let $P_{JM} < G$ be the parabolic subgroup with Lie algebra $\mathfrak{p}_{JM}$. The map
\[
    G \times_{P_{JM}} \mathfrak{n}_2 \longrightarrow \Obar_e \subset \mathfrak{so}_{2n}
\]
is known as the \emph{Jacobson--Morozov resolution}.

Using this resolution, we define the following moduli space: 
\[
\mathbf{Higgs}_{\Obar} = \left\{ (E, \theta, E_{P_{JM}}) \ \left| \
\begin{aligned}
    & E \ \text{is a principal } G\text{-bundle}, \\
    & \theta \in H^0(\Sigma, \Ad(E) \otimes \omega_{\Sigma}(x)), \\
    & E_{P_{JM}} \ \text{is a } P_{JM}\text{-reduction of } E \text{ at } x, \\
    & \Res_x \theta \in E_{P_{JM}} \times_{P_{JM}, \Ad} \mathfrak{n}_2
\end{aligned} \right. \right\}.
\]
Equivalently, $\mathbf{Higgs}_{\overline{\mathbf{O}}}$ parameterizes tuples $(E, g, \theta, \alpha_E, \Fil^{\bullet}_{P_{JM}})$, where:
\begin{itemize}
    \item $E$ is a vector bundle of rank $2n$,
    \item $g: E \otimes E \rightarrow \mathcal{O}_{\Sigma}$ is a symmetric nondegenerate pairing,
    \item $\alpha_E: \bigwedge^{2n} E \xrightarrow{\sim} \mathcal{O}_\Sigma$ is a framing,
    \item $\theta \in \operatorname{H}^0(\Sigma, \Ad(E) \otimes \omega_{\Sigma}(x))$ is the Higgs field,
    \item $\Fil^{\bullet}_{P_{JM}}$ is a filtration of the fiber $E|_x$ determined by $P_{JM}$,
\end{itemize}
such that the residue $\Res_x \theta$ preserves the filtration $\Fil^{\bullet}_{P_{JM}}$ and lies in the image of $\mathfrak{n}_2$ under the adjoint action. See \cite[\S 4.2]{WWW24} for more details.

In the following, when no confusion arises, we use the notation $(E, \theta)$ to refer to either a residually nilpotent Higgs bundle or an ordinary Higgs bundle.
 
\begin{remark}\
\begin{itemize}
    \item The moduli space $\bf{Higgs_{\Obar}}$ has two connected components, denoted by $\mathbf{Higgs}_{\Obar}^+$ and $\mathbf{Higgs}_{\Obar}^-$. We use $\mathbf{Higgs}_{\Obar}^\pm$ to denote either connected component.
    \item The moduli space $\mathbf{Higgs}_{\Obar}^\pm$ corresponds to the orbit itself rather than merely the associated partition. In particular, this means that, for very even orbit $\bf{O}=\bf{O}^I \bigsqcup \bf{O}^{II}$, $\mathbf{Higgs}_{\overline{\mathbf{O}}^I}^\pm$ and $\mathbf{Higgs}_{\overline{\mathbf{O}}^{II}}^\pm$ are distinct.
\end{itemize}
\end{remark}

One of the most important methods for studying the Higgs moduli space is via the Hitchin map, whose image is called the \emph{Hitchin base}. In types A, B, and C, the Hitchin base is typically given by the collection of all coefficients of the characteristic polynomial of the Higgs field.

However, in type D, this description is insufficient—the Hitchin base must include additional functions beyond the characteristic polynomial coefficients. This phenomenon, first observed by Tachikawa~\cite{Tac}, motivates the following definition.

\begin{definition}\label{Coulomb Hitchin base}
   We define the \emph{Coulomb Hitchin base} as the affinization:
    \[
        \bf{A}_{\Obar}:= \mathrm{Spec}\left(\CC\left[\mathbf{Higgs}_{\Obar}^\pm\right] \right).
    \] 
\end{definition}

Let $\mathbf{O} \subset \mathfrak{so}_{2n}$ be a nilpotent orbit with partition\footnote{In the very even case, the partition $\mathbf{d}$ corresponds to either the orbit $\mathbf{O}^{I}$ or $\mathbf{O}^{II}$. In both cases, we still use $\mathbf{d}$ to define $\eta_{2i}$.} 
\[ 
    \bf{d}=[d_1, \cdots ,d_r]=[\bf{T}_1, \cdots , \bf{T}_l].
\]
We first define the sequence $\eta = \{ \eta_{2i} \}_{i=1}^n$ by
\begin{align} \label{delta_sing}
    \eta_{2i}=\min \left\{ j \mid \sum_{k=1}^j d_{k}\geq 2i\right\}.
\end{align}
Note that $\eta_{2n}=r$ is even. 

Then, similar to \cite[Lemma 4.4]{WWW24}, we have the Hitchin map 
\[
    h_{\Obar}: \bf{Higgs}_{\Obar}^{\pm} \rightarrow \bf{H}
\]
which takes the coefficients of the characteristic polynomial with the last term given by the Pfaffian as an exception:
\[
    \bf{H}=\bigoplus_{i=1}^{n-1} \operatorname{H}^0\left(\Sigma, \omega_{\Sigma}^{2i} \otimes \mathcal{O}_{\Sigma}( (2i - \eta_{2i})x) \right)\bigoplus \operatorname{H}^0\left(\Sigma, \omega_{\Sigma}^{n} \otimes\mathcal{O}_{\Sigma}( (n - \tfrac{\eta_{2n}}{2})x )\right).
\]

Here we recall the construction of the Pfaffian. $(E, \theta)$ is a $\operatorname{SO_{2n}}$ Higgs bundle, which mean we not only have a nondegenerate pairing $g: E\otimes E\rightarrow \sO_{\Sigma}$, but also a framing $\alpha_{E}: \bigwedge^{2n}E\xrightarrow{\sim} \sO_{\Sigma}$. Notice that $g\circ\theta$ can be regarded as a global section of $\operatorname{H}^{0}(\Sigma, (\bigwedge^2E^\vee)\otimes \omega_{\Sigma}(x))$ then the $n$-th wedge of $g\circ \theta$ is a global section of $\operatorname{H}^{0}(\Sigma, (\bigwedge^{2n}E^\vee)\otimes \omega_{\Sigma}(x)^{\otimes n})$. Thus $(\alpha_E^\vee)^{-1}\circ \bigwedge^{n}(g\circ \theta)$ is a section in $\operatorname{H}^{0}(\Sigma, \omega_{\Sigma}(x)^{\otimes n})$ and this is the pfaffian of $\theta$.

The space of coefficients of the characteristic polynomial, denoted by $\mathbf{H}_{\Obar}$, is commonly referred to as \emph{Hitchin base}. In types A and C for all nilpotent orbits, and in type B for special nilpotent orbits, the space $\mathbf{H}_{\Obar}$, constructed analogously to $\mathbf{H}$ (excluding the Pfaffian), are affine spaces and we have
\[
    \mathbf{A}_{\Obar} = \mathbf{H}_{\Obar}.
\]

However, in type D, it may happen that $\mathbf{H}_{\Obar}$ is a proper subvariety of $\mathbf{H}$ and $\mathbf{A}_{\Obar} \neq \mathbf{H}_{\Obar}$. 

In the following, we will first describe generic fibers. In particular, we show that generic fibers are torsors of self-dual abelian varieties if and only if $\bf{O}$ is special. We then make use of the description of generic fibers to show that  $\mathbf{A}_{\Obar}$ is the normalization of $\mathbf{H}_{\Obar}$.
\subsection{Generic Fibers of $h_{\Obar}$}

In this subsection, we give a concrete description of the Hitchin fiber  $h_{\Obar}^{-1}(a)$ for generic $a \in \mathbf{H}_{\Obar}$ making use of spectral curves and the parabolic Beauville–Narasimhan–Ramanan (BNR) correspondence.

Consider $\pi : \Tot(\omega_{\Sigma}(x)) \rightarrow \Sigma$,
 where $\Tot(\omega_{\Sigma}(x))$ denotes the total space of the line bundle $\omega_{\Sigma}(x)$. Let
\[
    \lambda \in \operatorname{H}^0(\Tot(\omega_{\Sigma}(x)), \pi^* \omega_{\Sigma}(x))
\]
be the tautological section.

\begin{definition}
    Given $a = (a_2, a_4, \ldots, a_{2n-2}, p_n) \in \mathbf{H}_{\Obar}$, we define the \emph{spectral curve} $\Sigma_a$ to be the zero locus of the section
    \[
        \lambda^{2n} + a_2 \lambda^{2n-2} + \cdots + p_n^2 \in \operatorname{H}^0(\Tot(\omega_{\Sigma}(x)), \pi^* \omega_{\Sigma}^{2n}(2nx)).
    \]
    Denote by $\pi_a : \Sigma_a \rightarrow \Sigma$ the projection, and let $\sigma$ be the involution on $\Sigma_a$ induced by $\lambda \mapsto -\lambda$.
\end{definition}

The spectral curve $\Sigma_a$ is singular at the points where $p_n$ vanishes. We now define an open subset of $\mathbf{H}_{\Obar}$ that parametrizes generic spectral data.

\begin{definition} \label{Def:H^KL}
    Let $\mathbf{O}$ be a nilpotent orbit of type D. The open subset $\mathbf{H}^{\mathrm{KL}} \subset \mathbf{H}_{\Obar}$ consists of those $a \in \mathbf{H}_{\Obar}$ such that:
    \begin{itemize}
        \item[(a)] The section $p_n$ has only simple zeros away from the marked point $x$;
        \item[(b)] Near the marked point $x$, the local equation of $\Sigma_a$ is generic in $L\mathfrak{c}_{\bf{O}}$, i.e., it satisfies the genericity condition in Theorem \ref{Thm. decomposition}.
    \end{itemize}
    We also define $\mathbf{A}^{\mathrm{KL}}$ to be the inverse image of $\mathbf{H}^{\mathrm{KL}}$ under the map $\mathbf{A}_{\Obar} \rightarrow \mathbf{H}_{\Obar}$.
\end{definition}
\begin{remark}
     Later by showing the connectedness of generic Hitchin fibers, we show that the map $\mathbf{A}_{\Obar} \rightarrow \mathbf{H}_{\Obar}$ is generically one-to-one. Then we can identify $\mathbf{A}^{\mathrm{KL}}$ to $\mathbf{H}^{\mathrm{KL}}$.
\end{remark}
Let $\bar{\Sigma}_a$ denote the normalization of $\Sigma_a$, which inherits the involution $\sigma$. We denote the preimage of $x$ in $\bar{\Sigma}$ by $x_1, \cdots , x_r$, where $x_i$ corresponds to $e_i$ in the partition $\bf{deg}$. Denote the quotient curve by $\bar{\Sigma}_a/\sigma$ and we also use $\sR_a$ to denote the fixed points of $\sigma$, which we may also treat as a divisor. Thus $\sR_a=\sum x_j$ so that  $d_j\in \bf{T}_i$ of type D2.  Then we let 
\[
    \bar{\pi}_a: \bar{\Sigma}_a \rightarrow \Sigma
\]
be the natural map. We define the associated \emph{Prym variety} as the neutral component of
\[
    \Prym_a := \Prym(\bar{\Sigma}_a, \bar{\Sigma}_a/\sigma),
\]
and define
    \[
    \Prym(\sR_a):=\{\sL\in\Pic(\bar{\Sigma}_a) \mid \sigma^*\sL\cong\sL^{\vee}(\sR_a)\}.
    \]

Finally, we denote the fibers of $h_{\overline{\bf{O}}}^{\pm}$ over $a \in \mathbf{H}^{\mathrm{KL}}$ by
\[
    F_a^{\pm} := (h_{\Obar}^{\pm})^{-1}(a).
\]

\begin{proposition}\label{prop:the map l}
    We have a natural morphism $L_D: F_a^\pm\rightarrow \Prym(\sR_a)$.
\end{proposition}

\begin{proof}
    Consider $(E, \theta)\in F_a^\pm$, then we by Theorem \ref{Thm. decomposition}, we have a natural $\theta$-invariant subsheaf $E^\prime$ of $E$, so that the restriction of $\theta$ on the fiber $E^\prime|_{x}$ has partition $\bf{deg}$. Now we cover the base curve $\Sigma$ with a formal neighborhood around $x$ and $\Sigma \setminus \{x\}$. On one hand, by Theorem \ref{Thm. decomposition}, we see that locally around $x$, $E^\prime$ is isomorphic to a pushforward of a line bundle on the normalized spectral curve $\bar{\Sigma}_a$; on the other hand, over $\Sigma \setminus \{x\}$, by a similar argument as \cite[5.14]{Hit87S}, $E^\prime$ is also isomorphic to the push forward of a line bundle on $\bar{\Sigma}_a$. Glue all things together, we see that we have a line bundle $\sL$ on $\bar{\Sigma}_a$, so that $E^\prime\cong \bar{\pi}_{a*}\sL$.

     By the Grothendieck--Serre duality, we have
    \[
    (\bar{\pi}_{a*}(\sL))^{\vee}\cong \bar{\pi}_{a*}(\sL^{\vee} \otimes \sO_{\bar{\Sigma}_a}(\sR_a)).
    \]
    Since $\bar{\pi}_{a*}\sL$ is an orthogonal bundle we have
    \[
    \sigma^{*}\sL\cong\sL^{\vee}\otimes\sO_{\bar{\Sigma}_a}(\sR_a).
    \]
    Hence we have the desired morphism.
\end{proof}

We choose $x_0$ as the base point of the Abel-Jacobian map: 
 \[
     \bar{\Sigma}_a\rightarrow \Jac(\bar{\Sigma}_a),\quad x\mapsto \mathcal{O}_{\bar{\Sigma}_a}(x-x_0).
 \]
 Let $\mathcal{P}$ denote the Poincar\'e line bundle on $\bar{\Sigma}_a\times \Jac(\bar{\Sigma}_a)$. Restricting $\mathcal{P}$ to $\bar{\Sigma}_a\times \Prym_a$, which we still denote by $\mathcal{P}$, retains the symmetry induced by the involution $\sigma$, satisfying $\sigma^*\mathcal{P}\cong \mathcal{P}^\vee$. 
 
 For a point $x_i$, $1\leq i \leq r$, let $\mathcal{P}_{x_i}$ denote the restriction of $\sP$ to $\{x_i\}\times \Prym_a\cong\Prym_a$, in particular, $\sP_{x_0}$ is trivial. Moreover, $\sigma^*\sP_{x_i}\cong \sP_{x_j}$, if $\sigma(x_i)=x_j$. equivalently $\sP_{x_j} \cong \sP_{x_i}^{\vee}$, and if $x_i$ is fixed by $\sigma$, we have $\sP_{x_i} \in \Prym^\vee[2]$. Similarly, we also have such line bundles on $\Prym(\sR_a)$, which we also denote as $\sP_{x_i}$.

 \begin{lemma}
     If $\sigma(x_i)=x_j$, then we have a nondegenerate bilinear pairing on $\sP_{x_j} \oplus \sP_{x_i}$ and $\operatorname{OG}(1, \sP_{x_j} \oplus \sP_{x_i})$ is a trivial degree $2$ cover of $\Prym_a$.
 \end{lemma}

 \begin{proof}
     The nondegenerate bilinear pairing is given by the fact $\sP_{x_j} \cong \sP_{x_i}^{\vee}$. We consider $\operatorname{OG}(1, \sP_{x_j} \oplus \sP_{x_i})$ as a subvariety of $\mathbb{P}(\sP_{x_j} \oplus \sP_{x_i})$, then $\operatorname{OG}(1, \sP_{x_j} \oplus \sP_{x_i})$ is exactly the disjoint union of $0$ section and $\infty$ section hence disconnected.
 \end{proof}

 Now we want to give a concrete description of $F_a^{\pm}$ by stduying the fiber of $L_D$. That is, given an $\sL\in \Prym(\sR_a)$, we want to reconstruct residually nilpotent Higgs bundles in $F_a^{\pm}$.

 Firstly we set $E^\prime=\pi_{a*}\sL$, as we can see in the proof of Proposition \ref{prop:the map l}, the residually nilpotent Higgs bundle we want to construct should contains $E^\prime$ and should be a subsheaf of $(E^\prime)^\vee$. Consider $$0\longrightarrow E^\prime \longrightarrow (E^\prime)^\vee \longrightarrow Q\longrightarrow 0,$$ here $Q$ is a vector space of dimension $\beta(\bf{d})$ supporting at the marked point $x$ and it is $\oplus_jQ_j$ for all the $j$ such that $d_j\in\bf{T}_i$ of type D2. Thus the reconstruction of $E$ is equivalent to the local case as in the proof of Theorem~\ref{Thm. decomposition}.

 So the fiber of $
 L_D$ over $\sL$ is characterized by the sets of $\iota$-isotropic subspaces of $Q_i$ for each $\bf{T}_i$ of type D2, which gives a subvariety denoted by $\Prym_{W_i}(\sR_a)$, of an orthogonal Grassmanian bundle of $\Prym(\sR_a)$. In a similar way, we can construct a subvariety of an orthogonal Grassmanian bundle of $\Prym_a$, which we denote as $\Prym_{W_i, a}$.

 \begin{proposition}(\cite[Lemma 5.28, Proposition 5.29]{WWW24})
     The natural projection $\Prym_{W_i,a}\rightarrow \Prym_a$ is a finite morphism of degree $2^{\#\bf{T_i}-1}$ between abelian varieties whose kernel are 2-torsion points.
 \end{proposition}

\begin{definition}
    Denote the fiber product
    \[ \prod_{\Prym_a}^{\bf{T}_i \text{ of type D2}}\Prym_{W_i,a}\times_{\Prym_a}\Prym(\sR_a)
    \] 
    by $\Prym_{\Obar, a}(\sR_a)$. We have a finite cover $$\Prym_{\Obar, a}\longrightarrow \Prym_a$$ which is a \emph{connected component} of the similar fiber product  $\prod_{\Prym_a}\Prym_{W_i,a}$ coming from the natural morphism:
    \[
    \Prym^{\vee}_{\Obar,a}\rightarrow\Prym_{\Obar,a}
    \]
    where $\Prym^{\vee}_{\Obar,a}$ is the dual of $\Prym_{\Obar,a}$ and the morphism is induced from the natural polarization of prym varieties. In particular, each connected component of $\Prym_{\Obar, a}(\sR_a)$ is a torsor over $\Prym_{\Obar, a}$.
\end{definition}

\begin{proposition}\label{Prop: factor through}
    $\Prym_{\Obar, a}(\sR_a)$ has two connected components. The morphism $L_D$ factor through $\Prym_{\Obar, a}(\sR_a)\rightarrow \Prym(\sR_a)$. Moreover, $F_a^\pm$ is isomorphic to a connected component of $\Prym_{\Obar, a}(\sR_a)$.
\end{proposition}

\begin{proof}
    First note that the filtration corresponding to $P_{JM}$ is uniquely determined by nilpotent orbit if we fix the Levi type. Then, we only consider connected components of nilpotent orbits.
    
    If there exists a type D2 in the partition of $\bf{O}$, consider the points $x_j$ such that $d_j\in \bf{T}_i$ for some $\bf{T}_i$ of type D2, then such points are all the points on $\bar{\Sigma}$ fixed by $\sigma$. Now by \cite[ Lemma 5.7, Proposition 5.16, 5.29 ]{WWW24}, we see that $\Prym_{\Obar, a}(\sR_a)$ has two connected components. Similar as the third part of the proof of \cite[Proposition 5.26]{WWW24}, one can show that $F_a^\pm$ is isomorphic to a connected component of $\Prym_{\Obar, a}(\sR_a)$.

    When there is no type D2 in the partition of $\bf{O}$, the involution $\sigma$ has no fixed points. In this case $\Prym_{\Obar, a}(\sR_a)=\Prym(\sR_a)$ has two connected components. Thus we also have that $F_a^\pm$ is isomorphic to a connected component of $\Prym_{\Obar, a}(\sR_a)$.
\end{proof}

\begin{theorem}\label{thm.O fiber}
    For $a\in \bf{A}^{\text{KL}}$, the residually nilpotent Hitchin fibers $(h_{\Obar}^{\pm})^{-1}(a)$ is a torsor of $\Prym_{\Obar, a}$. As a consequence, $\bf{O}$ is special if and only if $(h_{\Obar}^{\pm})^{-1}(a)$ is a torsor of a self-dual abelian variety.
\end{theorem}

\begin{proof}
    Since $\mathbf{A}^{\mathrm{KL}}$ is identical to $\mathbf{H}^{\mathrm{KL}}$, thus by Proposition~\ref{Prop: factor through} we see that $(h_{\Obar}^\pm)^{-1}(a)$ is a torsor of $\Prym_{\Obar, a}$.

    By Lemma \ref{lem.special}, $\bf{O}$ is special if and only if each type D2 in the partition is of the form $[\alpha_1, \alpha_2]$ with $\alpha_1>\alpha_2$. Recall that $\Prym_{\Obar, a}$ is isomorphic to the fiber product of $\Prym_{W_i, a}$ for each $\bf{T}_i$ of type D2, hence by \cite[Lemma 5.12]{WWW24} and the construction of $\Prym_{W_i, a}$, $\Prym_{\Obar, a}$ is self-dual if and only each type D2 is of length $2$, which is equivalent to $\bf{O}$ being special.
\end{proof}

\subsection{The Coulomb Hitchin base}

The space $\mathbf{H}_{\Obar}$ turns out to be the zero locus of certain quadratic relations in $\mathbf{H}$, which we now show. This has been observed before, for example see \cite{BK18}.

Fix a local coordinate $t$ around the marked point $x$, and trivialize the canonical bundle $\omega_\Sigma$ via $dt$. This yields a surjective map 
\[
    \operatorname{ev}_x:\bf{H}\longrightarrow \mathbb{C}^n,
\]
defined by sending each section to the coefficient of $t^{\eta_{2i} - 2i}$ (or $t^{\eta_{2n}/2 - n}$ for the final term) in its local expansion in powers of $t$.

\begin{remark}
  The map $\operatorname{ev}_x$ depends on the choice of local coordinate $t$. However, as we will see later, the description of $\mathbf{H}_{\Obar}$ and $\mathbf{A}_{\Obar}$ are independent of this choice.   
\end{remark}

Consider each type D1* block $\mathbf{T}_i$ in the partition $\mathbf{d}$ such that
\[
    \mathbf{T}_{i-1} \neq \mathbf{T}_i = \cdots = \mathbf{T}_{i+k_i} \neq \mathbf{T}_{i+k_i+1}.
\]
Let $\mathbf{T}_{i+j} = [d_{i_1+2j}, d_{i_1+1+2j}]$ for $0\leq j \leq k_i$ so we have
\[
    d_{i_1} = d_{i_1+1} = \cdots = d_{i_1 + 2k_i + 1}.
\]

We now define a set of relations on $\mathbb{C}^n$, assuming $c_0 := 1$. Recall that there are $l$ many blocks in partition $\mathbf{d}$. Then:
\begin{itemize}
    \item[1.] If $i + k_i < l$, define
    \begin{equation} \label{equation1}
        p_i(\lambda) := \sum_{j=0}^{2k_i + 2} c_{\scriptscriptstyle{\sum_{s < i} d_s + j d_i}} \, \lambda^{2k_i + 2 - j}\quad \in \; \CC[\lambda],
    \end{equation}
    and require that $p_i(\lambda)$ is a square.

    \item[2.] If $i + k_i = l$, define
    \begin{equation} \label{equation2}
        p_i(\lambda) := \sum_{j=0}^{2k_i + 1} c_{\scriptscriptstyle{\sum_{s < i} d_s + j d_i}} \, \lambda^{2k_i + 2 - j} + c_{2n}^2\quad \in \; \CC[\lambda],
    \end{equation}
    and again require that $p_i(\lambda)$ is a square.
\end{itemize}

\begin{definition}\label{def:d and d^aff}
Let $\mathfrak{d}_{\mathbf{d}} \subset \mathbb{C}^n$ denote the common zero locus of the square-root relations given in \eqref{equation1} and \eqref{equation2}, associated to the type~D1* blocks in the partition $\mathbf{d}$. We define $\mathfrak{d}^{\mathrm{aff}}_{\mathbf{d}}$ to be the normalization of $\mathfrak{d}_{\mathbf{d}}$.
\end{definition}

\begin{example}\label{example}
    Before proceeding further, let us illustrate the construction with an example. 
    
    Consider the orbit $\mathbf{O}$ with partition $[d_1,d_2,d_3,d_4,d_5,d_6]=[4,4,3,3,2,2]$. For a generic $\theta \in L^+\mathfrak{g}$ such that $\theta(0)\in \mathbf{O}$, we assume that the characteristic polynomial of $\theta$ admits the following factorization: 
    \[
        \chi(\theta)=f_{1}(\lambda)f_{1}(-\lambda)f_{2}(\lambda)f_{2}(-\lambda)f_{3}(\lambda)f_{3}(-\lambda)
    \]
    where $f_i(\lambda)=\lambda^{d_{2i-1}}+\cdots +h_{i}$. Hence $a_i\in t\sO$ and we use $\overline{h_i}$ to denote the leading term of $a_i$ in $t$. 

    By a direct computation, we see that the leading term of coefficients of $\chi(\theta)$ can be stated as follows:
    \begin{table}[htbp]
    \centering
    \[
        \begin{tabular}{c|c|c|c|c|c|c}    
        \text{Term}       & $\lambda^{18}$ & $\lambda^{14}$     & $\lambda^{10}$     & $\lambda^4$             & $\lambda^2$                 & $\lambda^0$  \\ 
        \hline
        \text{Coefficient} & $1$            & $2\overline{h_1}$  & $\overline{h_1}^2$ & $\overline{h_1 h_2}^2$ & $2\overline{h_1 h_2}^2\overline{h_3}$ & $\overline{h_1 h_2 h_3}^2$
        \end{tabular}
    \]
    \end{table}

    By \eqref{equation1}, here we have $p_1(\lambda)=c_0\lambda^{2}+c_{4}\lambda+c_8$. Since $c_0=1$, $c_4=2\overline{h_1}$ and $c_8=\overline{h_1}^2$, we see that $p_1(\lambda)$ is a square and it is similar for $p_3(\lambda)$.

    Meanwhile, since $c_0=1$ and $c_{14}=\overline{h_1h_2h_3}$ is given by the Pfaffian. By our rule for $\mathfrak{d}^{\mathrm{aff}}_{\bf{d}}$, we see that $\sqrt{c_{8}}$ and $\sqrt{c_{14}}$ are used. In particular, $\mathfrak{d}^{\mathrm{aff}}_{\bf{d}}$ is isomorphic to an affine space in this case.

\end{example}

The following technical lemmas will be used to describe $\bf{H}_{\Obar}$ and $\bf{A}_{\Obar}$.

\begin{lemma}\label{coe of product}
    Let $f(\lambda) = \lambda^{N} + a_1 \lambda^{N-1} + \cdots + a_{N} \in \mathcal{O}[\lambda]$. Assume that $f(\lambda)$ admits a factorization $f(\lambda) = \prod_{i=1}^r f_i(\lambda)$, where each $f_i(\lambda) = \lambda^{e_i} + b_{i,1} \lambda^{e_i-1} \cdots + b_{i,e_i}$ is an Eisenstein polynomial in general position, and the exponents satisfy $e_1 \geq \cdots \geq e_r$. For $b \in \mathcal{O}$, let $\overline{b}$ denote the leading (i.e., lowest order) nonzero coefficient of $b$ in its expansion in powers of $t$. Then, for any $i$ such that $e_i > e_{i+1}$, we have
    \[
        \overline{a_{\sum_{j \leq i} e_j}} = \prod_{j \leq i} \overline{b_{j,e_j}}.
    \]
\end{lemma}

\begin{proof}
    The coefficient $a_{\scriptscriptstyle{\sum_{j \leq i} e_j}}$ corresponds to the term $\lambda^{N - \sum_{j \leq i} e_j}$ in the expansion of $f(\lambda)$. Each $f_j(\lambda)$ contributes a term $b_{j,k}$ at degree $\lambda^{e_j - k}$, so the total contribution to $a_{\sum_{j \leq i} e_j}$ comes from
    \[
        a_{\scriptscriptstyle{\sum_{j\leq i}e_j}}=\sum_{\{(k_j)\; \mid\; \sum_{j=1}^{r} k_j=\sum_{j\leq i}e_j\}}\prod_{j=1}^{r}b_{j,k_j}
    \] 
    Observe that $\operatorname{ord}_t \left( \prod_{j=1}^{r} b_{j,k_j} \right) = \#\{ j \mid k_j \neq 0 \}$, since each non-zero $k_j$ contributes a higher order term in $t$. Therefore, the dominant term in the expansion (i.e., the one with minimal $t$-order) is achieved by minimizing the number of nonzero $k_j$.

    Under the condition $e_i > e_{i+1}$, the only way to satisfy $\sum_{j=1}^r k_j = \sum_{j \leq i} e_j$ with the fewest non-zero $k_j$ is to take $k_j = e_j$ for $j \leq i$, and $k_j = 0$ for $j > i$. This gives:
    \[
        \prod_{j=1}^{r} b_{j,k_j} = \left( \prod_{j > i} b_{j,0} \right) \cdot \left( \prod_{j \leq i} b_{j,e_j} \right).
    \]
    Since each $f_j$ is Eisenstein, $b_{j,0} = 1$. Therefore, the lowest-order (i.e., leading) term of $a_{\scriptscriptstyle{\sum_{j \leq i} e_j}}$ is
    \[
        \overline{a_{\scriptscriptstyle{\sum_{j \leq i} e_j}}} = \prod_{j \leq i} \overline{b_{j,e_j}},
    \]
    as claimed.
\end{proof}

\begin{lemma} \label{Volume form}
    Let $V$ be a vector space over $\mathbb{C}$, equipped with a non-degenerate symmetric bilinear form $g : V \otimes V \rightarrow \mathbb{C}$, and suppose $\dim V = 2n$ is even. Then any maximal isotropic subspace $W \subset V$ induces a framing
    \[
        \alpha_W : \bigwedge^{2n} V \xrightarrow{\sim} \mathbb{C}.
    \]
    Moreover, if $W_1$ and $W_2$ are maximal isotropic subspaces lying in different connected components of the orthogonal Grassmannian $\operatorname{OG}(n, V)$, then
    \[
        \alpha_{W_1} = -\alpha_{W_2}.
    \]
\end{lemma}

\begin{proof}
    Let $W \subset V$ be a maximal isotropic subspace. Then we have a short exact sequence
    \[
        0 \longrightarrow W \longrightarrow V \longrightarrow V/W \longrightarrow 0.
    \]
    Since $\mathsf{W}$ is maximal isotropic and $g$ is non-degenerate, the induced pairing 
    \[
        W \otimes V/W \rightarrow \mathbb{C}
    \]
    is perfect. This pairing yields a canonical isomorphism
    \[
        \bigwedge^n W \otimes \bigwedge^n (V/W) \xrightarrow{\sim} \mathbb{C}.
    \]
    On the other hand, the exact sequence above provides a natural isomorphism
    \[
        \bigwedge^{2n} V \cong \bigwedge^n W \otimes \bigwedge^n (V/W).
    \]
    Composing these isomorphisms defines the volume framing $\alpha_W : \bigwedge^{2n} V \xrightarrow{\sim} \mathbb{C}$.

    Now, suppose $W_1$ and $W_2$ are maximal isotropic subspaces lying in different connected components of $\operatorname{OG}(n, V)$. Then there exists an element $A \in \operatorname{O}_{2n}(V) \setminus \operatorname{SO}_{2n}(V)$ such that $A(W_1) = W_2$. Since $A$ has determinant $-1$, it acts on $\bigwedge^{2n} V$ by $-1$, and therefore sends $\alpha_{W_1}$ to $-\alpha_{W_2}$. Hence,
    \[
        \alpha_{W_1} = -\alpha_{W_2},
    \]
    as claimed.
\end{proof}

\subsubsection{Veronese Ring of Degree 2 as the normalization}
First consider the following ring homomorphisms:
\begin{align}\label{eq:quadratic items}
    \mathbb{C}[c_0, c_1,\ldots, c_{2e}]&\xrightarrow{\phi} \mathbb{C}[y_0,y_1,\ldots, y_e]\\\nonumber
    c_m&\mapsto \sum_{0\le i\le m}y_{i}y_{m-i}
\end{align}
We denote the image by $Q_e$ which is a subring of $\mathbb{C}[y_0, y_1,\ldots, y_e]$. In particular, we have:
\[
(\sum_{i=0}^{e}y_i\lambda^i)^2=\phi(\sum_{i=0}^{2e}c_i\lambda^i).
\]
In the following, for simplicity, we may denote $c_m:=\sum_{0\le i\le m}y_{i}y_{m-i}$ by dropping $\phi$.
We put $V_{e}=\mathbb{C}[y_iy_j]_{0\le i\le j\le e}\subset \mathbb{C}[y_0,\ldots, y_e]$. Then $V_e$ is the so-called Veronese ring of degree 2, which is a direct summand of the polynomial ring, hence Cohen--Macaulay and normal. See also \cite{Sai94}.

\begin{lemma}\label{lem:single normalization}
    Consider $Q_e$ as above, then each $y_i$ is integral over $Q_e$ for $0\leq i \leq e$. In particular, $V_e$ (the Veronese ring of degree 2)  is the normalization of $Q_e$.
\end{lemma}
\begin{proof}
    The idea is to "compactify" the map in \eqref{eq:quadratic items}, then quasi-finiteness is sufficient. Consider the following morphism:
    \begin{align}\label{eq:compactified}
        \mathbb{P}^{e+1}&\xrightarrow{f} \mathbb{P}^{2e+1}\\\nonumber
        [t:y_0:\ldots:y_e]&\mapsto [t^2:c_0:\ldots:c_{2e}]
    \end{align}
    Recall that $c_m:=\sum_{0\le i\le m}y_{i}y_{m-i}$. We now verify the quasi-finiteness. First suppose that $c_0\ne 0$, then $y_0^2=c_0$ has 2 nonnegative solutions, and then for each solution, there is a unique $[y_1:\ldots: y_{e}]$ because we solve the linear equation inductively:
    \[
    y_0y_i=c_{i}-\ldots.
    \]
    For $c_0=0$, then $y_0=0$, we reduce to $e-1$ case. In particular, $f:\mathbb{P}^{e+1}\rightarrow f(\mathbb{P}^{e+1})$ is of generic degree 4.
    Now consider the following commutative diagram:
    \[
    \begin{tikzcd}
        \mathbb{A}^{e+1}\ar[r]&f^{-1}(\Spec Q_e)\ar[r]\ar[d]&\mathbb{P}^{e+1}\ar[d, "f"]\\
        &\Spec Q_e\ar[r]&f(\mathbb{P}^{e+1})
    \end{tikzcd}
    \]
    Notice that  $\mathbb{A}^{e+1}\hookrightarrow f^{-1}(\Spec Q_e)$ is defined by $t=0$. Hence we have $\mathbb{C}[y_0,\ldots, y_e]$ is finite over $Q_e$ of generic degree 2. In particular, $V_e$ is finite over $Q_e$.

    We now show that $\Frac(V_e)=\Frac(Q_e)$. Actually, we verify that $\frac{y_i}{y_0}\in Q_e[1/y_0^2]$. Clearly, it holds for $i=0,1$. Suppose this is true for $i\le n$. Consider:
    \[
    c_{n+1}=y_0y_{n+1}+..\in Q_e[1/y_0^2]
    \]
    then $\frac{y_{n+1}}{y_0}\in Q_e[1/y_0^2]$ by induction. Hence we have $V_e[1/y_0^2]=Q_e[1/y_0^2]$. Now the second statement follows since $V_e$ is normal.
\end{proof}

Recall that in Equations \eqref{equation1}, \eqref{equation2}, if we have two adjacent $\textbf{T}_i\neq \textbf{T}_{i+1}$ both of type D1*, the constant term of $p_i(\lambda)$ coincide with the leading coefficient of $p_{i+1}(\lambda)$. Hence we need a ``chained version" of the above normalization. It finally turns out that the Veronese ring of degree 2 is still the normalization.
\begin{Notation}
     We put $\vec{e}=\{e_1,e_2,\ldots, e_{\ell}\}$ with $e_i\ge 1$. Let $A_{\vec{e}}$ be the quotient of the polynomial ring $\mathbb{C}[y_{ij}]_{1\leq i \leq l,\ 0\leq j \leq e_i}$ by the ideal generated by $\{y_{i,e_i}-y_{i+1,0}\}$. Clearly, $A_{\vec{e}}$ is still a polynomial ring. We denote by $V_{\vec{e}}$ the Veronese ring of degree 2 of $A_{\vec{e}}$.
\end{Notation}

\begin{definition}\label{defn:chained Veronese}
   Now consider the subring of $\mathbb{C}[y_{ij}]_{1\leq i \leq l,\ 0\leq j \leq e_i}$  generated by monomials $\{y_{ij}y_{ik}\}_{1\leq i \leq l,\ 0\leq j\leq k \leq e_i}$ and then we call its image in $A_{\vec{e}}$ a \emph{chained Veronese ring of degree 2}, denoted by $CV_{\vec{e}}$. We call $\Spec CV_{\vec{e}}$ an \emph{affine chained Veronese variety of degree 2}. Similarly, we can define chained version of $Q_e$, denoted as $CQ_{\vec{e}}$.
\end{definition}

\begin{lemma}\label{lem:chained v2 is normalization}
    Let $CQ_{\vec{e}}$ be as above, 
    \begin{enumerate}
        \item $V_{\vec{e}}$ is the normalization of $CQ_{\vec{e}}$.
        \item The normalization of $CV_{\vec{e}}[y_{11}]$ and $CQ_{\vec{e}}[y_{i e_i}]$ for any $1\leq i \leq l$ are both $A_{\vec{e}}$, the polynomial ring.
    \end{enumerate}
\end{lemma}
\begin{proof}
     Following the same method as Lemma \ref{lem:single normalization}, we can show that $CV_{\vec{e}}$ is finite over $CQ_{\vec{e}}$ with same fractional field.

     Clearly $V_{\vec{e}}\supset CV_{\vec{e}}$. And for any element $y_{ij}y_{\ell m}\in V_{\vec{e}}$ but not in $CV_{\vec{e}}$, we show that it lies in $\Frac(CV_{\vec{e}})$. First, we assume that $\ell=i+1$. Then $y_{ij}y_{\ell m}=\frac{y_{ij}y_{ie_i}^2y_{\ell m}}{y_{ie_i}^2}\in\Frac(CQ_{\vec{e}})$ since in $A_{\vec{e}}$, we have $y_{ie_i}=y_{(i+1)0}$. The general case can be proved by finding a longer chain.
     
     Now we only need to prove that the Veronese subring of $V_{\vec{e}}$ is finite over $CV_{\vec{e}}$. Again we only need to prove that $y_{ij}y_{\ell m}$ are finite over $CV_{\vec{e}}$. But it follows easily, since $y_{ij}^2y_{\ell m}^2$ lies in the chained Veronese $CV_{\vec{e}}$.

     For the second statement, it follows from that $A_{\vec{e}}\subset V_{\vec{e}}[1/y_{00}]$ since $V_{\vec{e}}$ is the Veronese subring of degree 2.  
\end{proof}

\begin{proposition}\label{prop:dd affine}
    \quad 
    \begin{itemize}
        \item [(1)] If $\bf{O}$ is not very even, then $\mathfrak{d}^{\mathrm{aff}}_{\bf{d}}$ is isomorphic to product of an affine space and several affine Veronese varieties of degree 2. 
        \item [(2)] $\mathfrak{d}^{\mathrm{aff}}_{\bf{d}}$ is disconnected if and only if $\bf{O}$ is very even and in this case $\mathfrak{d}^{\mathrm{aff}}_{\bf{d}}$ is a disjoint union of two isomorphic affine spaces.
    \end{itemize}
\end{proposition}

\begin{proof}
    By Lemma \ref{lem:chained v2 is normalization}, we see that if $\bf{O}$ is not very even, then $\mathfrak{d}^{\mathrm{aff}}_{\bf{d}}$ is isomorphic to product of an affine space and several affine Veronese varieties of degree 2. And if $\bf{O}$ is very even, the existence of $c_{2n}$ will provide another square root of $c_0$, since in the very even case, $c_0$ and $c_{2n}$ are ``chained together", thus we have two connected components and each is isomorphic to an affine space.
\end{proof}

\begin{proposition} \label{d o bar}
    We have $ \mathbf{H}_{\Obar} = \operatorname{ev}_x^{-1}(\mathfrak{d}_{\bf{d}})$ and $\mathbf{A}_{\Obar} \rightarrow \mathbf{H}_{\Obar}$ is the normalization map. 

\end{proposition}

\begin{proof}
    To prove that the image of $h_{\overline{\bf{O}}}$ is contained in $\operatorname{ev}_x^{-1}(\mathfrak{d}_{\bf{d}})$, we proceed similarly to the proofs of  \cite[Proposition 4.7, 5.3]{BK18}, where they considered the case of Richardson orbits. The argument generalizes to arbitrary nilpotent orbits due to \cite[Proposition 6.4 ]{Spa88}.

    To prove that the image of $h_{\overline{\bf{O}}}^{\pm}$ is equal to $\operatorname{ev}_x^{-1}(\mathfrak{d}_{\bf{d}})$, we use Theorem~\ref{thm.O fiber} to see that the image of $h_{\overline{\bf{O}}}^{\pm}$ is dense in $\operatorname{ev}_x^{-1}(\mathfrak{d}_{\bf{d}})$ and by the properness of $h_{\overline{\bf{O}}}^{\pm}$, we see that $\bf{H}_{\Obar}=\operatorname{ev}_x^{-1}(\mathfrak{d}_{\bf{d}})$.

    By Theorem~\ref{thm.O fiber}, the generic fiber of $h_{\Obar}^{\pm}$ is smooth and connected. In particular, $(h_{\bar{\bf{O}}}^{\pm})_*\sO_{\bf{Higgs}_{\bar{\bf{O}}}^{\pm}}$ is finite and of generic rank 1 over the Hitchin base. Since $\bf{Higgs}_{\bar{\bf{O}}}^{\pm}$ is normal, $\bf{A}_{\bar{\bf{O}}}$ is the normalization of $\bf{H}_{\bar{\bf{O}}}$.
\end{proof}
Let $\operatorname{ev}_x^{-1}(\mathfrak{d}^{\mathrm{aff}}_{\bf{d}})$ denote the fiber product of $\operatorname{ev}_x^{-1}(\mathfrak{d}_{\bf{d}})$ and $\mathfrak{d}^{\mathrm{aff}}_{\bf{d}}$. As a corollary (also of Lemma \ref{lem:chained v2 is normalization}), we have

\begin{corollary}    
    The \emph{Coulomb Hitchin base} $\bf{A}_{\Obar}$ is a connected component of $\operatorname{ev}_x^{-1}(\mathfrak{d}^{\mathrm{aff}}_{\bf{d}})$ and is isomorphic to product of an affine space with several affine Veronese variety of degree 2.
\end{corollary}

\begin{proposition}
        The space $\operatorname{ev}_x^{-1}(\mathfrak{d}^{\mathrm{aff}}_{\bf{d}})$ is disconnected if and only if $\mathbf{O}$ is very even; in that case, it is a disjoint union of two isomorphic affine spaces corresponding to $\mathbf{O}^{I}$ and $ \mathbf{O}^{II}$. Moreover, $\bf{Higgs}_{\Obar^{I}}^{\pm}$ and $ \bf{Higgs}_{\Obar^{II}}^{\pm}$ are maped into different connected components of $\operatorname{ev}_x^{-1}(\mathfrak{d}^{\mathrm{aff}}_{\bf{d}})$.
\end{proposition}

\begin{proof}
    When $\bf{O}^I$ and $\bf{O}^{II}$ are very even with the same partition $\bf{d}$, we have $$\bf{Higgs}_{\Obar^{I}}^{\pm}\bigsqcup \bf{Higgs}_{\Obar^{II}}^{\pm}\longrightarrow \operatorname{ev}_x^{-1}(\mathfrak{d}^{\mathrm{aff}}_{\bf{d}}).$$ We construct a regular function on $\bf{Higgs}_{\Obar^{I}}^{\pm}$ and $\bf{Higgs}_{\Obar^{II}}^{\pm}$ in a uniform way, then show that such regular function will distinguish the image of $\bf{Higgs}_{\Obar^{I}}^{\pm}$ and $\bf{Higgs}_{\Obar^{II}}^{\pm}$. By Proposition \ref{prop:dd affine}, $\operatorname{ev}_x^{-1}(\mathfrak{d}^{\mathrm{aff}}_{\bf{d}})$ is disconnected if and only if $\mathbf{O}$ is very even and it is a disjoint union of two isomorphic affine spaces.

    Use Lemma~\ref{Volume form}, we can give the regular function as follows. Since the orbits $\bf{O}^I$ and $\bf{O}^{II}$ are very even, the corresponding JM-filtration associated to $P_{JM}$ will have subspaces which are maximal isotropic, moreover, for $\bf{O}^I$ and $\bf{O}^{II}$, the corresponding maximal isotropic subspaces lie in different components of $\operatorname{OG}(n, E|_x)$. So by Lemma~\ref{Volume form}, these two maximal isotropic subspaces gives two different framing on $E|_x$, use such two framings, one can show that the image of $\bf{Higgs}_{\Obar^{I}}^{\pm}$ and $\bf{Higgs}_{\Obar^{II}}^{\pm}$ are disjoint in $\operatorname{ev}_x^{-1}(\mathfrak{d}^{\mathrm{aff}}_{\bf{d}})$. To conclude, $\bf{A}_{\Obar^I}$ and $\bf{A}_{\Obar^{II}}$ equal different connected components of $\operatorname{ev}_x^{-1}(\mathfrak{d}^{\mathrm{aff}}_{\bf{d}})$.
\end{proof}

As a consequence, we can give a criterion of the smoothness of $\bf{A}_{\Obar}$.

\begin{corollary}\label{smoothness}
    The space $\mathbf{A}_{\Obar}$ is singular if and only if there exists a type D1* block $\mathbf{T}_i$ such that, for some indices $i_0 < i < i_1$, both $\mathbf{T}_{i_0}$ and $\mathbf{T}_{i_1}$ are not of type D1*. If $\mathbf{A}_{\Obar}$ is smooth, then it is isomorphic to an affine space.

    Similarly, the space $\mathbf{H}_{\Obar}$ is singular if and only if there exists a type D1* block in the partition. If $\mathbf{H}_{\Obar}$ is smooth, then it is isomorphic to an affine space.
\end{corollary}

\section{Parabolic Hitchin Systems: Richardson Orbits}\label{par Hitchin}

In this section, we study the parabolic Higgs moduli space where we insert the data of $T^*(G/P) \cong G \times_P \mathfrak{n}$ at the marked point. We denote this moduli space by $\mathbf{Higgs}_P$. As before, $\bf{Higgs}_P$ has two connected components and we use $\bf{Higgs}_P^{\pm}$ to denote each one component. Let the corresponding Richardson orbit be $\mathbf{O}_R$, and denote the generalized Springer map $T^*(G/P) \rightarrow \Obar_R$ by $\mu_P$, as before.

The Hitchin map for $\mathbf{Higgs}_P^{\pm}$ naturally takes values in $\mathbf{H}_{\Obar_R}$, which, as we have seen in the previous section, may be singular. However, we will show that in this case, the Hitchin map actually factors through a finite cover of $\mathbf{H}_{\Obar_R}$. This finite cover, denoted by $\mathbf{A}_P$, is the \emph{Coulomb Hitchin base}
\[
    \mathbf{A}_P := \mathrm{Spec}\left( \mathbb{C}\left[\mathbf{Higgs}_P^{\pm}\right] \right),
\]
and we will prove that it is isomorphic to an affine space.


\begin{definition}
\[
    \mathbf{Higgs}_P = \left\{ (E, \theta, E_P) \ \left| \
    \begin{aligned}
        & E \ \text{is a principal $G$-bundle}, \\
        & \theta \in H^0(\Sigma, \Ad(E) \otimes \omega_{\Sigma}(x)), \\
        & E_P \ \text{is a $P$-reduction of} \ E \ \text{at} \ x, \\
        & \Res_{x} \theta \in E_P \times_{P, \Ad} \mathfrak{n}
    \end{aligned}
    \right.\right\}.
\]
\end{definition}

\begin{remark}\label{rem:two P}
    Let $P$ be a parabolic subgroup of $\SO_{2n}$ with Levi type $(p_1, \ldots, p_k; q)$, i.e., we have an admissible filtration $\Fil^{\bullet}_{P}$:
    \begin{align*}
        \mathbb{C}^{2n} = F_0 \supset F_1 \supset F_2 \supset \ldots \supset F_{m} \supset F_{m}^\perp \supset \ldots \supset F_{1}^\perp \supset F_{0}^\perp = 0,
    \end{align*}
    such that 
    \begin{align*}
        \dim F_m/F_m^\perp = q, \quad \dim F_{i-1}/F_i = p_i \quad \text{for} \quad i=1, \ldots, m.  
    \end{align*}
    As mentioned before, if $q = 0$, there exist two non-conjugate parabolic subgroups $P$ and $P'$ whose associated filtrations have the same dimension data as above. Note that since $q = 0$, we have $F_k = (F_m)^\perp$, so $F_m$ is a maximal isotropic subspace of $\mathbb{C}^{2n}$.

    In the proof of Theorem~\ref{thm:Tac's conj}, we will show how to distinguish $\mathbf{Higgs}_P^{\pm}$ from $\mathbf{Higgs}_{P'}^{\pm}$.
\end{remark}

As in the case of $\mathbf{Higgs}_{\Obar_R}^{\pm}$, we have a natural morphism
\[
    h_P^{\pm} : \mathbf{Higgs}_P^{\pm} \longrightarrow \mathbf{H}_{\Obar_R}.
\]
We begin by describing the generic fiber of $h_P$.

\begin{proposition} \label{generic fiber of hORbar}
    For $a \in \mathbf{H}_{\Obar_R}^{\mathrm{KL}}$, there exists a natural morphism
    \[
        (h_P^{\pm})^{-1}(a) \longrightarrow (h_{\Obar_R}^{\pm})^{-1}(a)
    \]
    of degree $\deg \mu_P$. Moreover, $(h_P^{\pm})^{-1}(a)$ have $\deg \mu_P$ many connected components and each connected component maps isomorphically onto $(h_{\Obar_R}^{\pm})^{-1}(a)$.
\end{proposition}

\begin{proof}
    Consider $(\mathcal{E},\mathcal{E}_{P_{JM}}, \theta) \in (h_{\Obar_R}^{\pm})^{-1}(a)$ for $a \in \textbf{H}^{\text{KL}}$, we have $\Res_x(\theta) \in \bf{O}_{R}$. By the definition of Jacobson--Morozov resolution, we identify
    \[
        (h_{\Obar_R}^{\pm})^{-1}(a) = \{(E, \theta) \mid \chi(\theta) = a \}.
    \]
    The natural map
    \[
        (E, (E|_x)_{P}, \theta) \mapsto (E, \theta)
    \]
    relates $(h_{P}^{\pm})^{-1}(a)$ and $(h_{\Obar_R}^{\pm})^{-1}(a)$. From Proposition~\ref{prop.deg_mu_P}, the fibers are governed by \eqref{Spal_fib}.

    From Theorem~\ref{thm.O fiber}, the generic Hitchin fiber $(h_{\Obar_R}^{\pm})^{-1}(a)$ is a torsor over $\Prym_{\Obar_R,a}$, a connected component of
    \[
         \prod_{\Prym_a}^{\bf{T}_i\; \text{of type D2}} \Prym_{W_i,a}.
    \]
    Following arguments similar to those in \cite[Propositions 5.35, 5.38]{WWW24}, we conclude that $h^{-1}_{P}(a)$ is a torsor over
    \[
        \Prym_{\Obar_R,a} \times_{\Prym_a} \prod_{2i-1 \in I_P} \Prym_{x_{2i-1} - x_{2i}, a},
    \]
    where each $\Prym_{x_{2i-1} - x_{2i}, a}$ is defined as
    \[
        \Prym_{x_{2i-1} - x_{2i}, a} := \OG \left( \cO \oplus \cP_{x_{2i-1}} \otimes \cP_{x_{2i}}^{-1} \right).
    \]

    Finally, by the definition of $I(P)$, if $2i - 1 \in I(P)$, then $d_{2i-1}$ belongs to some block $\mathbf{T}_j$ of type D2. In this case, the fiber product
    \[
        \Prym_{\Obar_R,a} \times_{\Prym_a} \Prym_{x_{2i-1} - x_{2i}, a}
    \]
    has two connected components. Therefore, we conclude.
\end{proof}

In general, the Stein factorization theorem implies that the map $h_P^{\pm}$ factors through a finite cover of $\mathbf{H}_{\Obar_R}$. In the remainder of this section, we prove the following theorem, which was conjectured by Tachikawa~\cite{Tac}.

\begin{theorem}\label{thm:Tac's conj}
    The Hitchin map $h_{P}^{\pm}$ factors through the Coulomb Hitchin base:
    \begin{align*}
    \xymatrix{
        \bf{Higgs}^\pm_{P} \ar[dr]_{h_{P}^\pm} \ar[rr]^{h_{P,\rm{aff}}^\pm} &    & \bf{A}_{P} \ar[dl]^{f_P} \\
        & \bf{H_{\Obar_R}} &
    },
    \end{align*}
    such that 
    \begin{itemize}
        \item[1.] $\bf{A}_P$ is a connected component of $\operatorname{ev}^{-1}(\mathfrak{d}_P)$ (see Definition~\ref{def:d_P});
        \item[2.] $\bf{A}_P$ is isormorphic to an affine space;
        \item[3.] For any $a\in \bf{A}_P^{\KL}:=f_P^{-1}(\bf{H}^{\KL})$, $(h_{P,\rm{aff}}^\pm)^{-1}(a)$ is a torsor of self-dual abelian variety $\Prym_{\Obar_R, a}$.
    \end{itemize}
\end{theorem}

Let $\bar{S}$ be a scheme of finite type over the field $\mathbb{C}$ and assume that we have a family of parabolic $\SO_{2n}$-Higgs bundle $(\mathcal{E}_{\bar{S}}, \theta_{\bar{S}})$ on $\bar{S}\times \Sigma$, flat over $\bar{S}$. Then we have an induced morphism $\bar{S}\rightarrow \bf{H}_{\Obar_R}$. Let $S \subset \bar{S}$ be the preimage of $\bf{H}_{\Obar}^{\KL}$ and then write the base change of $(\mathcal{E}_{\bar{S}}, \theta_{\bar{S}})$ to be $(\mathcal{E}_{S}, \theta_{S})$. We choose a local coordinate $t$ around $x$ and then base change the family $(\mathcal{E}_S, \theta_S)$ to $S\times \mathcal{O}$ to get $(\tilde{\mathcal{E}_S}, \tilde{\theta_S})$. Then by 2.4 of Lemma \ref{lem:stru of Ti}, Lemma \ref{lem:factorization over non-closed field} and Corollary \ref{coro:decomposition over general field}, there is an open subvariety $S^\circ\subset S$, such that over $S^\circ$ and for each $2j-1\in I(P)$, $\Ker T_{j}(\tilde{\theta_{S^\circ}})$ is a $\tilde{\theta_{S^\circ}}$ direct summand of $\tilde{\mathcal{E}_{S^\circ}}$, where $T_{j}(\lambda)$ is a factor of $\chi(\tilde{\theta_{S^\circ}})$ defined in the proof of Theorem \ref{Thm. decomposition}. 

Thus the filtration on $\sE_{\bar{S}}$ associated to the parabolic subgroup $P$ induces a filtration on $\Ker T_{j}(\tilde{\theta_{S^\circ}})$ for each $2j-1\in I(P)$.

\begin{proposition}\label{prop:maximal isotropic}
    The induced filtration on $\Ker T_{j}(\tilde{\theta_{S^\circ}})$ for each $2j-1\in I(P)$ contains a maximal isotropic subbundle.
\end{proposition}

\begin{proof}
    This follows from Proposition~\ref{prop.deg_mu_P}.
    Indeed, by Corollary \ref{coro:decomposition over general field}, there is an induced nondegenerate pairing on $\Ker T_{j}(\tilde{\theta_{S^\circ}})$. Although, the pairing on $\Ker T_j(\widetilde{\theta}_{S_0})$ is different from the pairing for Jardan basis in Proposition~\ref{prop.deg_mu_P}, the proof can be applied here.

    For $2j-1 \in I(P)$, the subspace
    \[
        \bigoplus\limits_{1 \leq i \leq d_{2j-1}} \CC v \left(i, 2j-1 \right)  \oplus \bigoplus\limits_{1 \leq i \leq d_{2j}} \CC v(i ,2j) 
    \]
    stands for $\Ker T_j(\widetilde{\theta_{S_0}})$. Let $\tilde{\Fil_P^\bullet}$ be the induced filtration on $V$. From the proof of Proposition~\ref{prop.deg_mu_P}, we know
    \[
        \tilde{F_m^\perp} = \bigoplus\limits_{1 \leq i \leq \tfrac{d_{2j-1}-1}{2}} \CC v \left(i, 2j-1 \right)  \oplus \bigoplus\limits_{1 \leq i \leq \tfrac{d_{2j}-1}{2}} \CC v(i ,2j) \oplus V_i
    \]
    is the maximal isotropic subspace. Here, $V_i$ is defined above Proposition~\ref{prop.deg_mu_P}.
\end{proof}

Now we want to construct some new rational functions on $\bar{S}$. We first consider the restriction of the pairing $g$ to each $\Ker T_{j}(\tilde{\theta_{S^\circ}})$ for $2j-1\in I(P)$ to get $g_j$. Notice that $g_j$ takes values in $\mathcal{O}_{S^\circ}\subset \mathcal{O}_{S^\circ}\otimes \mathbb{C}[\![t]\!]$. But this only gives an $\operatorname{O_{d_{2j-1}+d_{2j}}}$ structure on $\Ker T_{j}(\tilde{\theta_{S^\circ}})$ since we do not have a framing on $\Ker T_{j}(\tilde{\theta_{S^\circ}})$. However, by Proposition \ref{prop:maximal isotropic} above, we have natural maximal isotropic sunbbundle of $\Ker T_{j}(\tilde{\theta_{S^\circ}})$, which gives an $\operatorname{SO_{d_{2j-1}+d_{2j}}}$ structure on $\mathcal{E}_{S^\circ}$ by Lemma~\ref{Volume form}.

Now we can consider $\operatorname{pf}(g_j\theta_j)$ for each $2j-1\in I(P)$, which is well-defined only when we consider the first nonzero term in terms of $t$ by Lemma \ref{Volume form}. These are regular functions on $S^\circ$, hence rational functions on $\bar{S}$. Thus we have defined rational functions on the moduli space $\bf{Higgs}_P^{\pm}$, which we still denote as $\operatorname{pf}(g_j\theta_j)$.

\begin{definition} \label{def:d_P}
Let $\mathfrak{d}_{\mathbf{d}}$ be the base associated with the partition $\mathbf{d}$ of the Richardson orbit $\mathbf{O}_R$ (see Definition~\ref{def:d and d^aff}), and let $I(P)$ be the index set determined by the polarization $P$.

We define a finite cover $\mathfrak{d}_P$ of $\mathfrak{d}^{\mathrm{aff}}_{\mathbf{d}}$ by adjoining certain square roots to the coordinate ring of $\mathfrak{d}_{\mathbf{d}}$ as follows:
\begin{itemize}
    \item For any $2j - 1 \in I(P)$ with $j > 1$, adjoin the square root $\sqrt{c_{\scriptscriptstyle{\sum_{k \leq 2j - 2} d_k}}}$.
    \item If $1 \in I(P)$, adjoin the square root $\sqrt{c_{\scriptscriptstyle{d_1 + d_2}}}$.
    \item For each polynomial $p_i(\lambda)$ appearing in \eqref{equation1} or \eqref{equation2}, if either the leading coefficient or the constant term of a square root of $p_i(\lambda)$ has already been adjoined, then adjoin all coefficients of that square root, collectively denoted by $\{y_i\}$.
\end{itemize}
The normalization of the affine variety defined by this enlarged function ring is denoted by $\mathfrak{d}_P$.
\end{definition}

\begin{example}
    Here we continue with Example~\ref{example}. For the orbit $\bf{O}$ with partition $\bf{d}=[d_1,d_2,d_3,d_4,d_5,d_6]=[4,4,3,3,2,2]$, it is Richardson and there is a unique polarization $P$ so that $I(P)=\{3\}$ for $d_3$. By our construct above, to define $\mathfrak{d}_P$, we firstly considered the square root $\sqrt{c_{8}}$ since $3\in I(P)$. Then we consider another square root $\sqrt{c_{8}}$ since $p_{1}(\lambda)$ has leading coefficient $1$. Hence we have two square root of $c_8$, which are differed by a sign. Thus, $\mathfrak{d}_P$ is isomorphic to a disjoint union of two isomorphic affine spaces. 

    Meanwhile, to see why we consider the square root of $c_8$, the construction of $\operatorname{pf}(g_2\theta_2)$ above imples that $\overline{h_2}$ is a rational function on the moduli space, which means $\sqrt{c_{8}}=\sqrt{c_{14}}/\overline{h_2}$ is a rational function on the moduli space. Here we use the fact that $\sqrt{c_{14}}$ is a regular function on the moduli space.  Since $c_8$ is regular, we see that $\sqrt{c_8}$ is also regular due to the normality of the moudli space.
\end{example}

\begin{proposition}\label{no relation and degree counting}
    \quad
    \begin{itemize}
        \item [(1)] $\mathfrak{d}_P$ is nonsingular and each connected components is isomorphic to an affine space.
        \item [(2)] $\mathfrak{d}_P$ is disconnected if and only if $\bf{T}_1$ is of type D1* or $1\in I(P)$. Moreover, in this case, $\mathfrak{d}_P$ has exactly two connected components.
        \item [(3)] The degree of a connected component of $\mathfrak{d}_P$ to $\mathfrak{d}_{\bf{d}}$ is $\deg \mu_P$.
    \end{itemize}
\end{proposition}

\begin{proof}
    Recall that by Propsition \ref{prop:dd affine}, $\mathfrak{d}^{\mathrm{aff}}_{\mathbf{d}}$ is isomorphic to a product of affine space and several affine Veronese varieties of degree 2. Then by Lemma \ref{lem:chained v2 is normalization}, (1) follows.

    Assume that $\bf{T}_1$ is of type D1*, then by Lemma~\ref{lem:stru of Ti}, for every $\bf{T}_j$ which is not of type D1*, we have $2j-1\in I(P)$. Notice that $\bf{T}_1$ will give a polynomial $p_1(\lambda)$ as in \eqref{equation1} or \eqref{equation2}, and the leading coefficient of $p_1(\lambda)$ is $1$, hence by our construction, the square root of the constant term of $p_1(\lambda)$ is a regular function of $\mathfrak{d}_P$. We assume that $\bf{T}_1=\cdots =\bf{T}_{i_1}\neq \bf{T}_{i_1+1}$. If $\bf{T}_{i_1+1}$ is not of type D1*, then by our construction, the square root of the constant term of $p_1(\lambda)$ is added to the definition of $\mathfrak{d}_P$. If $\bf{T}_{i_1+1}$ is of type D1*, then we proceed by induction, which also shows that the square root of the constant term of $p_1(\lambda)$ is added. Both cases provided an alternative square root of the constant term of $p_1(\lambda)$, hence $\mathfrak{d}_P$ has two connected components.

    The other case is that $1\in I(P)$ and by our construction, we know that one square root of $c_{\scriptscriptstyle{d_1+d_2}}$ is added. By Lemma \ref{lem:stru of Ti} we know that $\bf{T}_1\neq \bf{T}_2$. If $\bf{T}_2$ is also not of type D1*, then by our construction, another square root of $c_{\scriptscriptstyle{d_1+d_2}}$ is added. If $\bf{T}_2$ is of type D1*, then we also proceed by induction as above, thus another square root of $c_{\scriptscriptstyle{d_1+d_2}}$ is added. Hence in this case, $\mathfrak{d}_P$ also has two connected components.

    Besides these relations happen at $\bf{T}_1$, it is easy to see that there are no other relations in the definition of $\mathfrak{d}_P$. Hence the description of $\mathfrak{d}_P$ follows as claimed.

    The degree of $\mathfrak{d}_P\rightarrow \mathfrak{d}_{\bf{d}}$ is computed as follows. Notice that for each $2j-1\in I(P)$, the function $\sqrt{c_{\scriptscriptstyle{\sum_{k\leq 2j-2}d_k}}}$ ($\sqrt{c_{\scriptscriptstyle{d_1+d_2}}}$ in the case $j=1$) defines a double cover of $\mathfrak{d}_{\bf{d}}$. Now by Lemma~\ref{lem:stru of Ti}, for every $\bf{T}_{i_0}$ of type D1* such that 
    \[
    \mathbf{T}_{i_0-1} \neq \mathbf{T}_{i_0} = \cdots = \mathbf{T}_{i_0+k_i} \neq \mathbf{T}_{i_0+k_i+1}.
    \]
    If $i_0+k_i=l$, then we see that the functions in $\{y_i\}$ defined by \eqref{equation2} only defines a generically one-to-to cover of $\mathfrak{d}_{\bf{d}}$. If $\mathbf{T}_{i_0+k_i+1}$ is not of type D1*, then by Lemma~\ref{lem:stru of Ti}, we see that $2(i_0+k_i+1)-1\in I(P)$ and hence $\sqrt{c_{\scriptscriptstyle{\sum_{k\leq 2(i_0+k_i+1)-2}d_k}}}$ is considered in the definition of $\mathfrak{d}_P$. Notice that $c_{\scriptscriptstyle{\sum_{k\leq 2(i_0+k_i+1)-2}d_k}}$ is the constant term in \eqref{equation1}, hence the functions in $\{y_i\}$ defined by \eqref{equation1} only defines a generically one-to-to cover of $\mathfrak{d}_{\bf{d}}$ considering that $\sqrt{c_{\scriptscriptstyle{\sum_{k\leq 2(i_0+k_i+1)-2}d_k}}}$ already defines a double cover. And if $\mathbf{T}_{i_0+k_i+1}$ is of type D1*, then we proceed by induction, hence all the functions in $\{y_i\}$ only defines a generically one-to-to cover of $\mathfrak{d}_{\bf{d}}$, considering that the functions $\sqrt{c_{\scriptscriptstyle{\sum_{k\leq 2j-2}d_k}}}$ already defines a cover of $\mathfrak{d}_{\bf{d}}$.

    In summary, the map $\mathfrak{d}_P\rightarrow \mathfrak{d}_{\bf{d}}$ has degree $2^{\#I(P)}$ if $\bf{O}$ is not very even and has degree $2$ if $\bf{O}$ is very even. Due to the description of connected components of $\mathfrak{d}_P$ and Proposition~\ref{prop.deg_mu_P}, we conclude. 
\end{proof}

\begin{proof}[Proof of Theorem~\ref{thm:Tac's conj}]    
    By a similar argument as in Proposition \ref{d o bar}, we have the morphism $\bf{Higgs}_P^{\pm}\rightarrow \mathfrak{d}_{\bf{d}}$, we regard $c_{2i}$ as regular functions on $\bf{Higgs}_P^{\pm}$. Now we want to show that the morphism $\bf{Higgs}_P^{\pm}\rightarrow \mathfrak{d}_{\bf{d}}$ factors through $\mathfrak{d}_P$, which mean that we need to show that all $\{y_i\}$ and $\{\sqrt{c_{\scriptscriptstyle{\sum_{k\leq 2j-2}d_k}}}\}$ in the definition of $\mathfrak{d}_P$ are regular functions on $\bf{Higgs}_P^{\pm}$. We consider the following cases.

    The first case is that $\bf{T}_i$ and $\bf{T}_{i+1}$ are both of type D1* and $\bf{T}_i\neq \bf{T}_{i+1}$. This case has been discussed in the proof of Proposition \ref{d o bar}.
    
    The second case is that $\bf{T}_i$ is of type D1* while $\bf{T}_{i+1}$ is not. Then consider the polynomial $p_i(\lambda)$ defined by $\bf{T}_i$ as in \eqref{equation1}, the constant term of $p_i(\lambda)$ is given by $c_{\scriptscriptstyle{\sum_{s\leq 2i}d_s}}$. Then by Lemma \ref{lem:chained v2 is normalization}, if $\sqrt{c_{\scriptscriptstyle{\sum_{s\leq 2i+2}d_s}}}$ is a regular function on $\bf{Higgs}_P^{\pm}$ and $\sqrt{c_{\scriptscriptstyle{\sum_{s\leq 2i+2}d_s}}/c_{\scriptscriptstyle{\sum_{s\leq 2i}d_s}}}$ is rational functions on $\bf{Higgs}_P^{\pm}$, then the coefficients of square root of $p_i(\lambda)$ are regular functions on $\bf{Higgs}_P^{\pm}$.

    The third case is that $\bf{T}_i$ is not of type D1* while $\bf{T}_{i+1}$ is. Then $\bf{T}_{i+1}$ provides a polynomial $p_{i+1}(\lambda)$ as in \eqref{equation1} or \eqref{equation2}. The leading coefficient of $p_{i+1}(\lambda)$ is $c_{\scriptscriptstyle{\sum_{s\leq 2i}d_s}}$. If the coefficients of square root of $p_{i+1}(\lambda)$ are regular functions on $\bf{Higgs}_P^{\pm}$ and $\sqrt{c_{\scriptscriptstyle{\sum_{s\leq 2i}d_s}}/c_{\scriptscriptstyle{\sum_{s\leq 2i-2}d_s}}}$ is rational functions on $\bf{Higgs}_P^{\pm}$, then $\sqrt{c_{\scriptscriptstyle{\sum_{s\leq 2i-2}d_s}}}$ is a regular function on $\bf{Higgs}_P^{\pm}$.

    The last case is that $\bf{T}_i$ and $\bf{T}_{i+1}$ are both not of type D1*. Now if $\sqrt{c_{\scriptscriptstyle{\sum_{s\leq 2i+2}d_s}}}$ is a regular function on $\bf{Higgs}_P^{\pm}$ and $\sqrt{c_{\scriptscriptstyle{\sum_{s\leq 2i+2}d_s}}/c_{\scriptscriptstyle{\sum_{s\leq 2i}d_s}}}$ is rational functions on $\bf{Higgs}_P^{\pm}$, then $\sqrt{c_{\scriptscriptstyle{\sum_{s\leq 2i}d_s}}}$ is a regular function on $\bf{Higgs}_P^{\pm}$.

    If we have $2i-1\in I(P)$, by Lemma \ref{coe of product}, the rational functions $\operatorname{pf}(g_i\theta_i)$ we constructed before are exactly $\sqrt{c_{\scriptscriptstyle{\sum_{s\leq 2i}d_s}}/c_{\scriptscriptstyle{\sum_{s\leq 2i-2}d_s}}}$. Notice that $c_{2n}$ is a regular function on $\bf{Higgs}_P^{\pm}$, then we induct from $l$ to the smallest $j$ such that $2j-1\in I(P)$, by Lemma~\ref{lem:stru of Ti}, we see that all $\{y_i\}$ and $\{\sqrt{c_{\scriptscriptstyle{\sum_{k\leq 2j-2}d_k}}}\}$ in the definition of $\mathfrak{d}_P$ are regular functions on $\bf{Higgs}_P^{\pm}$ then the morphism $\bf{Higgs}_P^{\pm}\rightarrow \mathfrak{d}_{\bf{d}}$ factors through $\mathfrak{d}_P$. Hence the morphism $\bf{Higgs}_P^{\pm}\rightarrow \bf{H}_{\Obar_R}$ factors through $\operatorname{ev}^{-1}(\mathfrak{d}_P)$.

    For two non-conjugate parabolic subgroups with the same Levi type, we need to show that the corresponding moduli spaces map to different components of $\operatorname{ev}^{-1}(\mathfrak{d}_P)$. The case that $\mathfrak{d}_P$ having two connected components only happens when $\bf{T}_1$ is of type D1* or $1\in I(P)$. In both cases, the filtration associted to $P$ contains a maximal isotropic subspace, see Remark~\ref{rem:two P}. Thus, similar to the proof of Proposition~\ref{d o bar}, we see that $\bf{Higgs}_P^{\pm}$ maps to a connected component of $\operatorname{ev}^{-1}(\mathfrak{d}_P)$, which we denote as $\bf{A}_P$.

    Since the degree of $\bf{A}_P\rightarrow \bf{H}_{\Obar_R}$ is equal to $\deg \mu_P$, which is also the number of connected component of $(h_P^{\pm})^{-1}(a)$ for $a\in \bf{H}_{\Obar_R}^{\KL}$ by Proposition~\ref{generic fiber of hORbar}, we see that the generic fiber of $h_P: \bf{Higgs}_P^{\pm}\rightarrow \bf{A}_P$ is connected, hence a torsor of $\Prym_{\Obar_R, a}$. 

    Notice that by Proposition~\ref{no relation and degree counting}, $\bf{A}_P$ is isomorphic to an affine space, hence $\bf{A}_P$ is isormorphic to the affinization $\mathrm{Spec}\left(\CC\left[\bf{Higgs}_P^{\pm}\right]\right)$ by the nomality of $\bf{Higgs}_P^{\pm}$ and the connectness of the generic fiber of $h_P: \bf{Higgs}_P^{\pm}\rightarrow \bf{A}_P$.
 \end{proof}

\bibliographystyle{alpha}
\bibliography{ref}
\end{document}